\documentclass[]{imsart}
\usepackage{url}
\usepackage{enumerate}
\RequirePackage{amsthm,amsmath,amsfonts,amssymb}
\RequirePackage[numbers]{natbib}

\usepackage{graphicx}
\usepackage{amsfonts}
\usepackage{amsmath}
\usepackage{url}

\usepackage{bbm}
\usepackage{amssymb}
\newtheorem {Proposition}{Proposition}[section]
\newtheorem {Lemma}[Proposition] {Lemma}
\newtheorem {Theorem}[Proposition]{Theorem}
\newtheorem {Corollary}[Proposition]{Corollary}
\newtheorem {Remark}[Proposition]{Remark}

\def\log{\mathop{\rm log}\nolimits}
\def\N{\mathbb{N}}

\def\R{\mathbb{R}}
\def\C{\mathbb{C}} 
\def\E{\mathbb{E}}
\def\P{\mathbb{P}}
\def\x{\mathbf{x}}
\def\y{\mathbf{y}}

\def\X{\mathbf{X}}
\def\Y{\mathbf{Y}}
\def\E{\mathbb{E}}

\def\mP{\mathrm{P}}
\def\mQ{\mathrm{Q}}

\def\Cs{\mathcal{C}^{s}(\Omega)}
\def\Cso{\mathcal{C}^{s}_0(\Omega)}

\def\Co{\mathcal{C}^{s}(\Omega)\times \mathcal{C}^s_0(\Omega)}
\def\CoAlpha{\mathcal{C}^{\alpha}(\Omega)\times \mathcal{C}^{\alpha}_0(\Omega)}
\def\AP{\mathcal{A}_{\mP}}
\def\AQ{\mathcal{A}_{\mQ}}
\def\bAP{\bar{\mathcal{A}}_{\mP}}
\def\bAQ{\bar{\mathcal{A}}_{\mQ}}
\usepackage{natbib}
\usepackage{bbm}
\usepackage[T1]{fontenc}    % use 8-bit T1 fonts
\usepackage{lmodern}        % sin esto no funciona en linux
\usepackage{hyperref}       % hyperlinks
\usepackage{url}            % simple URL typesetting
\usepackage{booktabs}       % professional-quality tables
\usepackage{amsfonts}       % blackboard math symbols

\usepackage{nicefrac}       % compact symbols for 1/2, etc.
\usepackage{microtype}      % microtypography
\usepackage{lipsum}
\usepackage{graphicx}
\usepackage{color}
\graphicspath{ {./images/} }

\newcommand{\ER}{\color{black}}

\begin{document} 
%\dedicatory{}
\begin{frontmatter}
\title{Weak limits of entropy regularized Optimal Transport: potentials, plans and divergences}
%\title{A sample article title with some additional note\thanksref{t1}}
\runtitle{Weak limits of entropy regularized optimal transport}
\runauthor{Gonz\'alez-Sanz, Loubes and Niles-Weed}
%\thankstext{T1}{A sample additional note to the title.}

\begin{aug}
\author[A]{\fnms{Alberto}~\snm{González-Sanz}\ead[label=e1]{ag4855@columbia.edu}},
\author[M]{\fnms{Jean-Michel}~\snm{Loubes}\ead[label=e2]{loubes@math.univ-toulouse.fr}}
\and
\author[B]{\fnms{Jonathan}~\snm{Niles-Weed}\ead[label=e3]{jnw@cims.nyu.edu}}
%%%%%%%%%%%%%%%%%%%%%%%%%%%%%%%%%%%%%%%%%%%%%%
%% Addresses                                %%
%%%%%%%%%%%%%%%%%%%%%%%%%%%%%%%%%%%%%%%%%%%%%%
\address[A]{Department of Statistics,
Columbia University, New York, USA, \printead{e1}}

\address[M]{
IMT, Université de Toulouse, France.  \printead{e2}}

\address[B]{ Courant Institute and Center for Data Sciences, New York University, New York, USA, \printead{e3}}

\end{aug}

\begin{abstract}
This work deals with the asymptotic distribution of both potentials and couplings of entropic regularized optimal transport for compactly supported probabilities in $\R^d$. We first provide  the central limit theorem of the Sinkhorn potentials---the solutions of the dual problem---as a Gaussian process in $\Cs$. Then we obtain  the weak limits of the couplings---the solutions of the primal problem---evaluated on integrable functions, proving a conjecture of  Harchaoui, Liu, and Pal (2020).  %In both cases, their limit is a real Gaussian random variable.
Finally we consider  the  weak limit of the entropic Sinkhorn divergence  between $P$ and $Q$ under both assumptions $H_0:\ {\rm P}={\rm Q}$ or $H_1:\ {\rm P}\neq{\rm Q}$. Under $H_0$ the limit is a quadratic form applied to a Gaussian process in a Sobolev space, while under $H_1$, the limit is Gaussian. We provide also a different characterisation of the limit under $H_0$ in terms of an infinite sum of an i.i.d.\ sequence of standard Gaussian random variables.  Such results enable statistical inference based on entropic regularized optimal transport.
\end{abstract}



\end{frontmatter}

\maketitle
\section{Introduction}
Optimal transport has proven its effectiveness as a powerful tool in statistical data analysis. Formulated as a minimization problem, it reads
$$
	\mathcal{T}_2(\mP,\mQ)=\min_{\pi\in \Pi(\mP,\mQ)} \int_{\R^d\times \R^d} {\textstyle \frac{1}{2}}\|\x-\y\|^2 d\pi(\x,\y),
$$
where $\Pi(\mP, \mQ)$ denotes the set of couplings between the probabilities $\mP$ and $\mQ$. Optimal transport provides a notion of discrepancy between distributions \cite[Chapter~7]{villani2003topics} useful for testing similarity between probabilities  \cite{BarGinMat99,BarGinUtz05,gonzalezdelgado2021twosample,del2019central}  and making inference. For this,  it is necessary to know the weak limit of $	\mathcal{T}_2(\mP_n,\mQ_m)$ when the empirical measures $\mP_n$ and $\mQ_m$ are used in the place of the population ones, $\mP$ and $\mQ$. Unfortunately, in general dimension---with the exception of perhaps a few simplified cases  \cite{delbarrio2021centraldisc,KlaTamMun20}---the limit is unknown.  Moreover, the rate is  slower than the usual parametric rate ${n}^{-\frac{1}{2}}$ \cite{weedBach,Fournier2013OnTR}. This  motivates the use of regularization methods for optimal transport, since they are not affected by the curse of  dimension, such as entropic regularization  \cite{Cut13}, or simplified versions,  such as sliced optimal transport \cite{RabinSliced}. Regularized optimal transport is now used for many practical applications  such as domain adaptation \cite{Courty2017OptimalTF}, counterfactual explanations  \cite{delara2021transportbased}, music transcription  \cite{MusicFlam},  diffeomorphic registration \cite{DiffeosOT,de2022diffeomorphic} and measure colocalization in super-resolution images  \cite{KlaTamMun20}. 

\indent Thought weak limits are known for the optimal value in the regularized optimal transport problem, less is known about the distributional limits of the  optimizers themselves. In this paper we provide the limits of the empirical solutions of the primal problem (plans/couplings), the dual problem (potentials), and the celebrated  Sinkhorn divergence \cite{GenPeyCut18}. 

Let $\Omega\subset\R^d$ be a compact set. The entropic regularized optimal transport cost between two probability measures $\mP,{\rm Q}\in \mathcal{P}(\Omega)$ is defined as the solution to the optimization problem 
\begin{equation}\label{kanto_entrop}
S_{\epsilon}(\mP,\mQ)=\min_{\pi\in \Pi(\mP,\mQ)} \int_{\R^d\times \R^d} {\textstyle \frac{1}{2}}\|\x-\y\|^2d\pi(\x,\y)+\epsilon H(\pi \mid \mP \otimes \mQ),
\end{equation}
where the relative entropy between two probability measures $\alpha$ and $\beta$ is written as $H(\alpha \mid \beta) = \int \log(\frac{d\alpha}{d\beta}(x))d\alpha(x)$ if $\alpha$ is absolutely continuous with respect to $\beta$, $\alpha\ll \beta$, and $+\infty$ otherwise. We denote the solution of \eqref{kanto_entrop} by $\pi_{\mP,\mQ}$.
This problem can also be written in its dual formulation 
\begin{equation}\label{dual_entrop}
  S_{\epsilon}(\mP,{\rm Q})= \sup_{\substack{f\in L_1(\mP) \\ g\in L_1({\rm Q})}} \int f(\x) d \mP(\x) + \int g(\y) d \mQ(\y) -\epsilon \int \int  e^{\frac{f(\x)+ g(\y)- \frac{1}{2}\|\x-\y\|^2}{\epsilon}} d \mP(\x) d \mQ(\y)  +\epsilon.
\end{equation}
There exists a unique pair $(f_{\mP,\mQ},g_{\mP,\mQ})$ of solutions of the above optimization problem such that 
\begin{equation}
    \label{eq:expected0}
    \int  g_{\mP,\mQ}(\x) d\mQ(\y)=0.  
\end{equation}
Moreover, it holds that $\pi_{\mP,\mQ}=\xi_{\mP,\mQ} d (\mP \otimes \mQ)$ with
\begin{align}\label{optimallityCodt01}
   \xi_{\mP,\mQ}(\x,\y) =e^{\frac{f_{\mP,\mQ}(\x)+g_{\mP,\mQ}(\y)- \frac{1}{2}\|\x-\y\|^2}{\epsilon}}. 
   \end{align}
 It can be shown that this optimality condition implies the  unique extension of the solutions of \eqref{optimallityCodt01} to the space $\mathcal{C}^{\alpha}(\Omega)\times \mathcal{C}^{\alpha}(\Omega)$, given by the relations 
\begin{align}
\begin{split}
    \label{optimallityCodt0}
       f_{\mP,\mQ}&=-\epsilon\log\left(\int e^{\frac{g_{\mP,\mQ}(\y)-\frac{1}{2}\|\cdot-\y\|^2}{\epsilon}}d\mQ(\y)\right)\quad {\rm and}\\
    g_{\mP,\mQ}&=-\epsilon\log\left(\int e^{\frac{f_{\mP,\mQ}(\x)-\frac{1}{2}\|\x-\cdot\|^2}{\epsilon}}d\mP(\x)\right).
\end{split}
\end{align}
These relations have an important consequence---the optimization class can be reduced from   $L_1({\rm Q})$ to $\mathcal{C}^{s}(\Omega)$, with, moreover, uniformly (for all $\mP,{\rm Q}\in \mathcal{P}(\Omega)$) bounded derivatives, see \cite{genevay2019,Weed19}. Since the class $\mathcal{C}^{s}(\Omega)$, with an appropriate choice of $ s$, is uniformly Donsker (eg. Section~2.7.1. in \cite{Vart_Well}), one can obtain the bound
$ \sqrt{n}\, \E\left[\mid S_{\epsilon}({\rm P}_n,{\rm Q})- S_{\epsilon}(\mP,{\rm Q})\mid\right] \leq C_\Omega,$
where the constant $C_\Omega$ depends polynomially on $\operatorname{diam}(\Omega)$, see \cite{Weed19}. Moreover, since the convergence of  $\E[S_{\epsilon}({\rm P}_n,{\rm Q})] $ towards its population counterpart occurs at a faster rate  (see, e.g., \cite{del2022improved}) and the fluctuations are asymptotically Gaussian  (\cite{delbarrio2019,Weed19,delbarrio2021central}), the weak limit
\begin{equation}\label{previousPaper}
    \sqrt{n}(S_{\epsilon}({\rm P}_n,{\rm Q})-S_{\epsilon}(\mP,{\rm Q}))\xrightarrow{w} N(0,\operatorname{Var}_{\X\sim 
    \mP}(f_{\mP,\mQ} (\X)))
\end{equation}
holds. Though \eqref{previousPaper} gives a weak limit for $S_{\epsilon}({\rm P}_n,{\rm Q}) $,  it  does not give information about the coupling $\pi_{\mP,\mQ}$ or the optimal dual variables $(f_{\mP,\mQ},g_{\mP,\mQ})$.
\vskip .1in

For statistical inference, obtaining the asymptotic behaviour of optimal regularized couplings  themselves, i.e., the limits of 
\begin{equation}
    \label{eq:etalim}
  \sqrt{\frac{n\, m}{n+m}}\int \eta\, (d\pi_{\mP_n,\mQ_m} -d\pi_{\mP,\mQ} ), \quad \text{where $\eta\in L^\infty(\mP\otimes\mQ)$,}
\end{equation}
is highly desirable. Indeed, in many of the previously cited applications, the optimal coupling itself is the object of interest.
Obtaining the limit of~\eqref{eq:etalim} would allow the statistician to obtain consistent confidence intervals for inference on the coupling.
    With regards to the limit of \eqref{eq:etalim}, the recent work \cite{ChaosDecom}  studied a modified regularization procedure inspired by Schrodinger’s lazy gas experiment giving rise to a different empirical estimator $\overline{\pi}_{\mP_n,\mQ_m} $. That work showed that if $\frac{m}{n + m} \to \lambda \in (0, 1)$, then $\overline{\pi}_{\mP_n,\mQ_m} $ satisfies  \begin{equation}\label{limitetaIntro}
        \sqrt{\frac{n\, m}{n+m}}\int \eta  (d \overline{\pi}_{\mP_n,\mQ_m} -d \pi_{\mP,\mQ})  \overset{w}{\longrightarrow} N(0,\sigma^2_{\lambda, \epsilon}(\eta)), \quad \eta\in L^2(\mP\otimes\mQ),
    \end{equation}
    where the variance $\sigma^2_{\lambda, \epsilon}(\eta)$ is 
    \begin{multline*}
        \lambda \operatorname{Var}_{\X\sim \mP}\left( (1-\AQ\AP)^{-1}\big(\eta_{\x} -\AQ\eta_{\y} \big)(\X)\right)\\+ (1-\lambda )\operatorname{Var}_{\Y\sim \mQ}\left( (1-\AP\AQ)^{-1}\big(\eta_{\y} -\AP\eta_{\x} \big)(\Y)\right),
    \end{multline*}
see  section~\ref{sec:tclPot} for the precise definitions of the operators $\AP,\AQ$ and  section~\ref{sec:TCLcouplings} for the ones of $\eta_{\x} ,\eta_{\y} $. Moreover,  the authors of \cite{ChaosDecom} conjectured that the distributional limit \eqref{limitetaIntro} holds also for the classic Sinkhorn regularization.
 \cite{Gunsilius2021MatchingFC} proved that \eqref{eq:etalim} is tight and the limit is centered, however the conjecture remained opened. In Theorem~\ref{Theorem:TCLcouplings} we prove that \eqref{limitetaIntro} holds for compactly supported measures, and therefore the conjecture of  \cite{ChaosDecom} is true. 

Theorem~\ref{Theorem:TCLcouplings} is derived as a consequence of the first-order linearization of the potentials, described in Theorem~\ref{Theorem:TCLpot}, whose proof is based on a reformulation of the optimality conditions \eqref{optimallityCodt0} as a  $Z$-estimation problem (see  for instance \cite{Vart_Well}). Differentiating in the Fr\'echet sense the objective function and using the uniform bounds provided by \cite{del2022improved}, the problem is reduced to the continuity and existence of the following operator in $\Cso\times\Cs$
\begin{equation*}
    \left(\begin{array}{cc}
  (1-\AQ\bAP)^{-1} & -(1-\AQ\bAP)^{-1}\AQ  \\
      - \bAP (1-\AQ\bAP)^{-1} & (1-\bAP\AQ)^{-1} 
   \end{array}\right),
\end{equation*}
which follows from Fredholm alternative \cite[Theorem 6.6]{Brezis} (note  that a similar result of invertibility of these operators between different Banach spaces has been obtained by \cite{CarlierFredho}). As a consequence, Theorem~\ref{Theorem:TCLpot} yields the limits, if $m=m(n)\to \infty$ and $\frac{m}{n+m}\to \lambda\in (0,1)$,
\begin{multline*}
     { \sqrt{\frac{n\, m}{n+m}}} {\textstyle\small \left(\begin{array}{c}
     f_{\mP_n,\mQ_m} -f_{\mP,\mQ}   \\
        g_{\mP_n,\mQ_m} -g_{\mP,\mQ} 
   \end{array}\right)}\\
   \xrightarrow{w} \epsilon {\textstyle\small\left(\begin{array}{c}
        {\textstyle\tiny      {\textstyle\tiny \sqrt{\lambda}}(1-\AQ\AP)^{-1}\AQ {\bf i}_{\xi_{\mP, \mQ}} \mathbb{G}_{\mP}-\sqrt{1-\lambda}}(1-\AQ\AP)^{-1} {\bf i}_{\xi_{\mQ, \mP}}\mathbb{G}_{\mQ}\\
          {\textstyle\tiny \sqrt{1-\lambda}}\AP (1-\AQ\AP)^{-1} {\bf i}_{\xi_{\mQ, \mP}}\mathbb{G}_{\mQ} -{\textstyle\footnotesize\sqrt{\lambda}}(1-\AP\AQ)^{-1} {\bf i}_{\xi_{\mP, \mQ}}\mathbb{G}_{\mP}  \end{array}\right)}, 
\end{multline*}
weakly in $ \Co$, where
$\mathbb{G}_{\mP}$ and $ \mathbb{G}_{\mQ}$ are independent $\mP$ and $\mQ$-Brownian bridges. 
 Moreover, in the one-sample case; 
\begin{equation*}
    \sqrt{n}\left(\begin{array}{c}
     f_{\mP_n,\mQ} -f_{\mP,\mQ}  \\
        g_{\mP_n,\mQ} -g_{\mP,\mQ} 
   \end{array}\right)\to -\left(\begin{array}{c}
         (1-\AQ\AP)^{-1}\AQ \mathbb{G}_{\mP,s} \\
        - (1-\AP\AQ)^{-1} \mathbb{G}_{\mP,s}  \end{array}\right),
\end{equation*}
weakly in $ \Co$. 
Theorem~\ref{Theorem:TCLpot},  apart from being interesting in itself, has many applications since, among other things, the derivative of $f_{\mP,\mQ} $ may be used to estimate the transport map from $\mP$ to $\mQ$  (eg. \cite{Pooladian2021EntropicEO,Rigollet2022OnTS}), which is also a useful tool for inference.
\\

The regularized transport cost is easier to compute than the usual optimal transport cost but is unsuitable for two-sample testing, since $S_{\epsilon}({\rm P},{\rm P})\neq 0 $. In~ \cite{GenPeyCut18}, the authors proposes to remedy this deficiency by defining the quadratic Sinkhorn divergence:
$$D_{\epsilon}(\mP,{\rm Q})= S_{\epsilon}(\mP,{\rm Q})-{\textstyle \frac{1}{2}}\left(  S_{\epsilon}(\mP,\mP)+ S_{\epsilon}({\rm Q},{\rm Q}) \right).$$
This definition satisfies several atractive properties: it is symmetric in $\mP,{\rm Q}$ and  
$D_{\epsilon}(\mP,{\rm Q})\geq 0$, with $D_{\epsilon}(\mP,{\rm Q})= 0$ if and only if $\mP={\rm Q}$   \cite[Theorem 1]{Feydy2019InterpolatingBO}. This quantity is therefore  an effective way to measure discrepancies between distributions. 

The weak limit of the empirical Sinkhorn's divergence, described in Theorem~\ref{Theorem:TCLDiv}, has different rates depending on the hypotheses $H_0:\ {\rm P}={\rm Q}$ or $H_1:\ {\rm P}\neq{\rm Q}$. Under $H_1$ the limit can be derived by means of  Efron-Stein linearization \cite{delbarrio2019,Weed19,delbarrio2021central,gonzalezdelgado2021twosample}, giving  
\begin{equation*}
    \sqrt{\frac{n\,m}{n+m}}(D_{\epsilon}(\mP_n,\mQ_m)- D_{\epsilon}(\mP,\mQ))\xrightarrow{w} N(0,\lambda\operatorname{Var}_{\mP}(\psi _{\mP,\mQ})+(1-\lambda)\operatorname{Var}_{\mQ}(\psi _{\mQ,\mP})),
\end{equation*}
and 
\begin{equation*}
    \sqrt{n}(D_{\epsilon}(\mP_n,\mQ)- D_{\epsilon}(\mP,\mQ))\xrightarrow{w} N(0,\operatorname{Var}_{\mP}(\psi _{\mP,\mQ})),
\end{equation*}
where $\psi_{\mP,\mQ} =f_{\mP,\mQ} - \frac{1}{2}(f_{\mP,\mP} +g_{\mP,\mP} ) $ and $\psi_{\mQ,\mP} =f_{\mQ,\mP} - \frac{1}{2}(f_{\mQ,\mQ} +g_{\mQ,\mQ} ) $.
Under  $H_0$, however, $\psi_{\mP,\mQ} =\psi_{\mQ,\mP} =0$, so the limit is degenerate, and in fact $D_{\epsilon}(\mP_n,\mQ)=\mathcal{O}_{\P}(\frac{1}{n})$. To obtain a non-trivial limit, we therefore conduct a second order analysis.  The limit, in this case, behaves as 
         $n\,D_1({\rm P}_n,{\rm P})\approx  \frac{1}{4}\sum_{j=1}^{\infty}\lambda_i^2 N_j^2$
where $ \{N_i\}_{i\in \N}$ is a sequence of i.i.d.\ random variables with $N_i\sim  N(0,1)$ and   $\{\lambda_{j}\}_{j\in \N}\subset [0, \infty)$ is a square-summable  sequence depending on $\mP$ and $\epsilon$.   A similar expression was previously known to hold only in the case where $P$ is discrete~\cite{BigotCLT}. Extending those results to the general case is not a trivial process (cf. \cite[Remark~8]{Goldfeld2022LimitTF}), and verifying the necessary conditions to obtain this limit is the most technical part of our argument.

\subsection{State-of-art}
In recent years there has been a substantial body of work studying the weak limits of the optimal transport problem. Since, for obvious reasons we cannot cite every one of them, we refer to \cite{HunKlaSta22} for a comprehensive survey. Focusing on entropic regularized optimal transport, similar results for finitely supported measures are obtained by \cite{KlaTamMun20} and \cite{BigotCLT}. The limits of the regularized cost \eqref{previousPaper} has been proven first in \cite{del2022improved} and \cite{Weed19}---using the Efron-Stein linearization--and then in \cite{Goldfeld2022StatisticalIW}---using Hadamard linearization. 

The conjecture of \cite{ChaosDecom} has been also investigated in~\cite{Gunsilius2021MatchingFC}, under slightly different assumptions. That work shows the existence of a weak limit for \eqref{eq:etalim}, which they conjecture is Gaussian. We prove their conjecture. They also derive a similar result to Theorem~\ref{Theorem:TCLpot}, but in a weaker norm. We prove convergence in the space $\Co$, which allows us to derive the limit of the Sinkhorn divergence.

Concurrently and independently of our work, \cite{Goldfeld2022LimitTF}  derive also the limits of the Sinkhorn optimal transport potentials  potentials and divergences. In that work the first order development is done by means of the functional Hadamard differentiability of the potentials. It is worth mentioning that  since a preprint of this work first appeared online, \cite{gonzalezsanzShayan2023weak} has generalized our results for non-smooth bounded costs. To obtain this generalization \cite{gonzalezsanzShayan2023weak} linearizes the potentials in the space $L^2({\rm P}_n)$ instead of in $C^{\alpha}(\Omega)$. This allows them to avoid  suprema  over  function  classes, which, for non-smooth costs, are not Donkser. 
\subsection{Outline of the paper}
The paper is organized as follows. The notation is given in Section~\ref{notation}.  The central limit theorem  for the regularized optimal transport potentials and its proof can be found in Section~\ref{sec:tclPot}; however, auxiliary Lemmas are proven in Section~\ref{sec:proofLemmas}. In Section~\ref{sec:TCLcouplings} we enunciate and prove the central limit theorem for the couplings, Theorem~\ref{Theorem:TCLcouplings}, and its immediate consequence, Corollary~\ref{Corollary:TCLcouplings}. Section~\ref{sec:divergences} deals with the weak limits of the Sinkhorn divergence, formally stated in Theorems~\ref{Theorem:TCLDivG} and \ref{Theorem:TCLDiv} . Also in this section the reader can find a discussion about the simplification of the limit by embedding it into a Hilbert space and the proof of the main result.  Auxiliary results and their proofs  are postponed to Section~\ref{auxiliary}.
\subsection{Notations}\label{notation}
For the reader's convenience, this section sets the notation used throughout this work. Unless otherwise specified, the probabilities ${\rm P}$ and ${\rm Q}$  are supported in the compact set  $\Omega\subset \R^d$, meaning that ${\rm P},{\rm Q}\in \mathcal{P}(\Omega)$. In the proofs we will use the notations 
\begin{align}
     \begin{split}\label{notationForProofs}
         &  f_{n,m}(\x)= \frac{f_{\mP_n,\mQ_m}(\x)}{\epsilon}, \quad g_{n,m}(\y)= \frac{ g_{\mP_n,\mQ_m}(\y)}{\epsilon},\quad \mathbb{C}(\x, \y)=e^{-\frac{1}{2\epsilon}\|\x-\y\|^2}, \\
         &  \quad f_*(\x)= \frac{f_{\mP,\mQ}(\x)}{\epsilon}, \quad g_{*}(\y)= \frac{f_{\mP,\mQ}(\y)}{\epsilon},
     \end{split}
 \end{align}
and 
$$ h_{n,m}(\x, \y)= f_{n,m}(\x)+g_{n,m}(\y),\quad h_{*}(\x, \y)= f_*(\x)+g_*(\y)$$
to reduce the size of the displays. 
For two probabilities $\mu,\nu\in \mathcal{P}(\Omega)$ the pair $(f_{\mu,\nu} ,g_{\mu,\nu} )$ is one solving \eqref{dual_entrop} for $(\mu,\nu)$, its direct sum ($(\x,\y)\mapsto f_{\mu,\nu} (\x)+ g_{\mu,\nu} (\y)$) is denoted by $h_{\mu,\nu} $. The solution of \eqref{kanto_entrop} is denoted by $\pi_{\mu,\nu} $ and its density with respect to  the direct product measure $\mu \otimes \nu$ by $\xi_{\mu,\nu} $. 

We set $s=\left\lceil \frac{d}{2}\right\rceil+1$, and denote, for any $\alpha\in \N$,  the space of  all functions on $\Omega$ that possess uniformly bounded partial derivatives up to order $\alpha$ as $\mathcal{C}^{\alpha}(\Omega)$, in which we consider the  norm
$$\|f \|_{\mathcal{C}^{\alpha}(\Omega)}= \|f \|_{\alpha}=\sum_{i=0}^{\alpha}\sum_{  \mid\beta  \mid= i}\|D^{\beta} f \|_{\infty}.$$
In the product space  $\mathcal{C}^{\alpha}(\Omega)\times \mathcal{C}^{\alpha}(\Omega)$  we consider the norm 
$$\|(f,g) \|_{\mathcal{C}^{\alpha}(\Omega)\times \mathcal{C}^{\alpha}(\Omega)}= \|(f,g) \|_{\alpha\times \alpha}= \|f \|_{\alpha}+ \|g\|_{\alpha}.$$
We define the spaces
$$ \mathcal{C}^{\alpha}_0(\Omega)=\left\{ f\in \mathcal{C}^{\alpha}(\Omega):\quad \mQ(f)=0 \right\} \ {\rm and}\  L^2_0(\mQ)=\left\{ f\in L^2(\mQ):\quad \int f d \mQ=0 \right\}. $$
Moreover, for a general Banach space $\mathcal{H}$, the norm is denoted as $\|\cdot\|_{\mathcal{H}}$. The operator norm of an operator $F: \mathcal{H}_1\to \mathcal{H}_2 $ is denoted as 
$ \|F \|_{\mathcal{H}_1\to \mathcal{H}_2}$. Unless otherwise stated, the random vectors $\X$ and $\Y$ are independent and follow  respectively the distributions $\rm P$ and $\rm Q$.
For a measurable function $f:\Omega\rightarrow\R$, the following expressions are equivalent:
\begin{equation*}
    \E[f(\X)]=\E_{\X\sim {\rm P}}[f(\X)]=\int f(\x)d {\rm P}(\x)=\int fd {\rm P}={\rm P}(f).
\end{equation*}
The function 
$\y\mapsto \int f(\x,\y)d {\rm P}(\x) $ is denoted by 
$\int f(\x,\cdot)d {\rm P}(\x)$.
Given a random sequence $\{w_n\}_{n\in}\subset \mathcal{H}$  and a real random sequence $\{a_n\}_{n\in\N}$,  the notation  $w_n=o_{\P}(a_n)$ means that the sequence $w_n/a_n$ tends to $0$ in probability, and  $w_n=\mathcal{O}_{\P}(a_n)$ that $w_n/a_n$ is stochastically bounded, for further details see Section~2.2 in \cite{vaart_1998}. Finally, for a Borel measure  $\mu$ in $\R^d$,  $L^2(\mu)$ denotes the space of square integrable functions. The spacial case of the Lebesgue measure $\ell_d$ in $\Omega$ is denoted by $L^2(\Omega)$. Probability measures are sometimes viewed as elements of the dual space of $\mathcal{C}^{\alpha}(\Omega)$, namely $(\mathcal{C}^{\alpha}(\Omega))'$,  with the standard dual norm 
$ \|P\|_\alpha'=\sup_{\|f\|_\alpha\leq 1}|P(f)|.$
\section{Central Limit Theorem of Sinkhorn potentials}\label{sec:tclPot}
The main result of this section is Theorem~\ref{Theorem:TCLpot}, which gives the first-order linearization of the difference between the empirical and population Sinkhorn potentials. As a consequence,  Corollary~\ref{Corollary:TCLpot}  gives the central limit theorem of that difference with rate $\sqrt{\frac{n\,m}{m+n}}$.   This section contains also the proof of Theorem~\ref{Theorem:TCLpot} and Corollary~\ref{Corollary:TCLpot}. The proofs of auxiliary results are postponed to the Appendix. 

Set two probabilities $\mP,\mQ\in \mathcal{P}(\Omega)$, recall that $\xi_{\mP,\mQ} $ is the density of $\pi_{\mP,\mQ} $ w.r.t to  $\mP \otimes\mQ$ and define  the operators
\begin{align}\label{operatorsA}
    \begin{split}
    \mathcal{A}_{\mP} :L^2({\rm P})\ni f&\mapsto \int \xi_{\mP,\mQ} (\x,\cdot)f(\x)d\mP(\x)\in \mathcal{C}^{\alpha}(\Omega) , \\ \AQ:L^2_0({\rm Q})\ni g&\mapsto \int \xi_{\mP,\mQ} (\cdot,\y)g(\y)d\mQ(\y)\in \mathcal{C}^{\alpha}(\Omega),\\
    \bAP:L^2_0({\rm P})\ni f&\mapsto\AP f-\int \AP f(\y)d\mQ(\y)\in \mathcal{C}^{\alpha}(\Omega)
    \end{split}
\end{align}
and,  for a function $h\in \mathcal{C}^{\infty}(\Omega^2)$, 
\begin{align*}
    {\bf i}_{h}: (\mathcal{C}^{\alpha}(\Omega))'&\to \mathcal{C}^{\infty}(\Omega)\\
   \nu &\mapsto  \left(  \y \mapsto \nu(h(\cdot, \y))\right),
\end{align*}
which for Radon measures $\nu\in (\mathcal{C}(\Omega))'$ takes the form 
$  {\bf i}_{h}(\nu)=\int h(\x, \cdot) d\nu(\x). $
Note that by convention, ${\bf i}_h$ always corresponds to the application of $\nu$ to the first coordinate of $h$.
At times during our argument, we will wish to apply this operation to either the first or second coordinate of $\xi_{\mP, \mQ}$ (or to its empirical counterpart), which we will indicate by swapping the order of the subscripts, e.g.,
\begin{align*}
	{\bf i}_{\xi_{\mQ, \mP}}(\nu) := \left(  \x \mapsto \nu(\xi_{\mP, \mQ}(\x, \cdot))\right)\,.
\end{align*}

The following result shows this operator is bounded. For a multiindex $b=(b_1, \dots, b_d)\in \N^d$ We use the notation
$ \partial_{b, \y } h= \partial_{(b_1, \dots, b_d, 0, \dots, 0) } h $. 
\begin{Lemma}\label{Lemma:ixi}
For all $b\in \N^{d}$ and $\nu\in \mathcal{C}^{\alpha}(\Omega)$, it holds that      $ \partial_b {\bf i}_{h}(\nu)={\bf i}_{\partial_{b, \y } h} (\nu) $, where $ \partial_{b, \y } h(\y, \x) $ denotes the partial derivative (in the standard multi-index notation) with respect to the $\y$'s component. Moreover, 
    for all $\beta\in \N$  there exists a constant $ C(\beta)  $ such that
$$\|{\bf i}_{h} (\nu)\|_{\mathcal{C}^{\beta}(\Omega)}\leq C(\beta)\|\nu\|_{(\mathcal{C}^\alpha(\Omega))'}, $$
so that ${\bf i}_{h}: (\mathcal{C}^{\alpha}(\Omega))'\to \mathcal{C}^{\beta}(\Omega)  $ is a well-defined bounded operator for any $\beta\in \N$.  
\end{Lemma}

Remark first that, given the optimal entropic coupling $\pi_{\mP,\mQ} $, the operators $\AP$ and $\AQ$ correspond to  the  barycentric projection  of the a generic function $f$ in $L^2({\rm P})$.  Our definition if therefore a natural generalization of the one proposed in \cite{Pooladian2021EntropicEO}, which  studies the projection  $\int \xi_{\mP,\mQ} (\cdot,\y) \y d\mQ(\y)$. $\bAP$ is nothing more than the  centered  version of the operator $\AP$. Note that in our work we only need to consider this operator but we could have defined in the same way $\bAQ $.
 In the following result such operators appear naturally when describing  the limit of the Sinkhorn potentials.
      \begin{Theorem}\label{Theorem:TCLpot}
      Let $\Omega\subset{\R^d}$ be a compact set, ${\rm P}, {\rm Q}\in \mathcal{P}(\Omega)$ and ${\rm P}_n$ (resp. $\mQ_m$) be the empirical measure of the i.i.d.\ sample $\X_1,\dots,\X_n$ (resp. $\Y_1,\dots,\Y_m$) distributed as $\rm P$ (resp. $\mQ$). If $m=m(n)\to \infty$ and $\frac{m}{n+m}\to \lambda\in (0,1)$,
\begin{multline*}
     \left(\begin{array}{c}
        f_{\mP_n,\mQ_m} - f_{\mP,\mQ}   \\
        g_{\mP_n,\mQ_m} - g_{\mP,\mQ}   
    \end{array}\right)
    =  \epsilon\, \left(\begin{array}{c}
       (1-\AQ\AP)^{-1}\AQ {\bf i}_{\xi_{\mP, \mQ}}(\mP_n-\mP)-(1-\AQ\AP)^{-1} {\bf i}_{\xi_{\mQ, \mP}}(\mQ_m-\mQ)     \\
         \AP (1-\AQ\AP)^{-1} {\bf i}_{\xi_{\mQ, \mP}}(\mQ_m-\mQ)  -(1-\AP\AQ)^{-1} {\bf i}_{\xi_{\mP, \mQ}}(\mP_n-\mP)\end{array}\right)\\+o_{\P}\left(\sqrt{\frac{n+ m }{n\,m}}\right)
\end{multline*}
 in $ \mathcal{C}^{\alpha}(\Omega)\times\mathcal{C}^{\alpha}(\Omega)$, for any $\alpha\in \N$. 
      \end{Theorem}
\begin{Remark}
   Recall that our notational convention implies that ${\bf i}_{\xi_{\mP, \mQ}}(\mP_n-\mP)$ and ${\bf i}_{\xi_{\mQ, \mP}}(\mQ_m-\mQ) $ implicitly refer to integration of $\xi_{\mP, \mQ}$ with respect to different coordinates.
\end{Remark}
By proving that
\begin{equation}
    \label{limitGaussian}
    \left(\begin{array}{c}
{\bf i}_{\xi_{\mP, \mQ}}  (\mP_n-\mP)\\ {\bf i}_{\xi_{\mQ, \mP} } (\mQ_m-\mQ)
\end{array}\right)=\left(\begin{array}{c}
   \frac{1}{n} \sum_{k=1}^n \xi_{\mP,\mQ} (\X_k,\cdot)-\E\left[ \xi_{\mP,\mQ} (\X, \cdot)\right]\\
  \frac{1}{m} \sum_{k=1}^m \xi_{\mP,\mQ} (\cdot, \Y_k)-\E\left[ \xi_{\mP,\mQ} (\cdot, \Y)\right]\\
\end{array}\right) 
\end{equation}
satisfies the central limit theorem in $\mathcal{C}^{\alpha}(\Omega)\times \mathcal{C}^{\alpha}(\Omega)$ and the linearization given in Theorem~\ref{Theorem:TCLpot} we obtain the limit behavior.  There are several ways to prove \eqref{limitGaussian}: the first involves utilizing the fact that the empirical process satisfies the central limit theorem in $(\mathcal{C}^{\alpha}(\Omega))'$ ($\alpha\geq s$), followed by the application of Lemma~\ref{Lemma:BoundOnA}. The second method involves using embeddings of Sobolev spaces into $(\mathcal{C}^{\alpha}(\Omega))'$ that are of type II and thus satisfy the central limit theorem. The last approach is to demonstrate that $\|{\bf i}_{\xi_{\mP, \mQ}}  (\mP_n-\mP)\|_{\mathcal{C}^{\alpha+1}(\Omega)}=\mathcal{O}_\P(1/n)$ thereby establishing that 
${\bf i}_{\xi_{\mP, \mQ}}  (\mP_n-\mP)$ is tight in $\mathcal{C}^{\alpha}(\Omega)$.   
We  leave the details for the reader.
   \begin{Corollary}\label{Corollary:TCLpot}
Let $\Omega\subset{\R^d}$ be a compact set, ${\rm P}, {\rm Q}\in \mathcal{P}(\Omega)$ and ${\rm P}_n$ (resp. $\mQ_m$) be the empirical measure of the i.i.d. sample $\X_1,\dots,\X_n$ (resp. $\Y_1,\dots,\Y_m$) distributed as $\rm P$ (resp. $\mQ$). Set $\alpha\in \N$ and suppose that both samples are mutually independent. Then, if $m=m(n)\to \infty$ and $\frac{m}{n+m}\to \lambda\in (0,1)$,
\begin{equation*}
   {\textstyle\tiny \sqrt{\frac{n\, m}{n+m}}} {\textstyle\small \left(\begin{array}{c}
     f_{\mP_n,\mQ_m} -f_{\mP,\mQ}   \\
        g_{\mP_n,\mQ_m} -g_{\mP,\mQ} 
   \end{array}\right)}\to \epsilon {\textstyle\small\left(\begin{array}{c}
        {\textstyle\tiny      {\textstyle\tiny \sqrt{\lambda}}(1-\AQ\AP)^{-1}\AQ {\bf i}_{\xi_{\mP, \mQ}} \mathbb{G}_{\mP}-\sqrt{1-\lambda}}(1-\AQ\AP)^{-1} {\bf i}_{\xi_{\mQ, \mP}}\mathbb{G}_{\mQ}\\
          {\textstyle\tiny \sqrt{1-\lambda}}\AP (1-\AQ\AP)^{-1} {\bf i}_{\xi_{\mQ, \mP}}\mathbb{G}_{\mQ} -{\textstyle\footnotesize\sqrt{\lambda}}(1-\AP\AQ)^{-1} {\bf i}_{\xi_{\mP, \mQ}}\mathbb{G}_{\mP}  \end{array}\right)}, 
\end{equation*}
weakly in $ \mathcal{C}^{\alpha}(\Omega)\times\mathcal{C}^{\alpha}_0(\Omega)$, where
$\mathbb{G}_{\mP}$ and $ \mathbb{G}_{\mQ}$ are independent $\mP$ and $\mQ$-Brownian bridges. 
 Moreover, in the one-sample case; 
\begin{equation*}
    \sqrt{n}\left(\begin{array}{c}
     f_{\mP_n,\mQ} -f_{\mP,\mQ}   \\
        g_{\mP_n,\mQ} -g_{\mP,\mQ} 
   \end{array}\right)\to -\epsilon\left(\begin{array}{c}
         (1-\AQ\AP)^{-1}\AQ {\bf i}_{\xi_{\mP, \mQ}} \mathbb{G}_{\mP}\\
        - (1-\AP\AQ)^{-1} {\bf i}_{\xi_{\mP, \mQ}} \mathbb{G}_{\mP} \end{array}\right),
\end{equation*}
weakly in  $ \mathcal{C}^{\alpha}(\Omega)\times\mathcal{C}^{\alpha}_0(\Omega)$.
   \end{Corollary}

   \begin{proof}[Proof of Theorem~\ref{Theorem:TCLpot}] We can assume that $\alpha\geq s$ without loss of generality.  In order to simplify the expressions we define the bilinear operator $ \mathbb{A}_\C $ as
\begin{align*}
    \mathbb{A}_\C:  (\mathcal{C}^{\alpha}(\Omega))' \times \mathcal{C}^{\alpha}(\Omega)   &\to \mathcal{C}^{\alpha}(\Omega)\\
    (f, \nu)& \mapsto \left(  \y \mapsto \nu( f(\cdot)\,  \C(\cdot, \y))\right), 
\end{align*} 
which for measures $\nu\in (\mathcal{C}(\Omega))'$ takes the form 
$$  \mathbb{A}_\C(f, \nu)=\int f(\x) \C(\x, \cdot) d\nu(\x). $$
The relation \eqref{optimallityCodt0} means that  $(f_*, g_*),\ (f_{n,m}, g_{n,m})\in \CoAlpha$ are just the solutions of $\Psi(f_*, g_*)=\Psi_{n,m}(f_{n,m}, g_{n,m})=0$, where $$    \Psi ,  \Psi_{n,m}:\CoAlpha \rightarrow \CoAlpha$$ are respectively defined as
\begin{align*}
\text{\small$ \Psi \left(   \begin{array}{c}
         f  \\
         g
    \end{array}\right)= \left(\begin{array}{c}
    f+ \log\left(\mathbb{A}_\C\left(e^{g}, \mQ \right) \right)
    \\
        g+ \log\left(\mathbb{A}_\C\left(e^{f}, \mP \right)\right)-\mQ\left( \log\left(\mathbb{A}_\C\left(e^{f}, \mP \right) \right)\right)
    \end{array}\right)$}
\end{align*}
and 
\begin{align*}
\text{\small$ \Psi_{n,m} \left(   \begin{array}{c}
         f  \\
         g
    \end{array}\right)= \left(\begin{array}{c}
    f+ \log\left(\mathbb{A}_\C\left(e^{g}, \mQ_m \right) \right)
    \\
        g+ \log\left(\mathbb{A}_\C\left(e^{f}, \mP_n \right)\right)-\mQ\left( \log\left(\mathbb{A}_\C\left(e^{f}, \mP_n \right) \right)\right)
    \end{array}\right)$}.
\end{align*}
 Note that we are subtracting the expectation w.r.t. ${\rm Q}$ in the second component so that the image of $  \Psi$ and $  \Psi_{n,m}$ lies in $\CoAlpha$.
The following estimates, derived from
Lemmas~\ref{Lamma_exponential} and \ref{Lemma:uniform_conv},  are fundamental:
\begin{align} 
    \left\| \mathbb{A}_\C(e^{f_{*}}, \mP)-\mathbb{A}_\C(e^{f_{n,m}}, \mP) \right\|_\alpha &=\mathcal{O}_\P\left(\sqrt{\frac{n+m}{n\,m}}\right) \label{Lamma_exponentialEquationOperator1},\\
    \left\| \mathbb{A}_\C(e^{f_{*}}, \mP)-\mathbb{A}_\C(e^{f_{n,m}}, \mP_n) \right\|_\alpha &=\mathcal{O}_\P\left(\sqrt{\frac{n+m}{n\,m}}\right) \label{Lamma_exponentialEquationOperator2},\\
     \left\| \mathbb{A}_\C(e^{f_{*}}, \mP)-\mathbb{A}_\C(e^{f_{*}}, \mP_n) \right\|_\alpha &=\mathcal{O}_\P\left(\sqrt{\frac{n+m}{n\,m}}\right) \label{Lamma_exponentialEquationOperator3}.
\end{align}
The same proof of Lemma~\ref{Lemma:ixi} gives the following result. 
\begin{Lemma}\label{Lemma:BoundOnA}
There exists a constant $C$ such that 
for each $\nu\in (\mathcal{C}^{\alpha}(\Omega))'$ and $f\in \mathcal{C}^{\alpha}(\Omega)$ 
$$ \|\mathbb{A}_\C(f, \nu)\|_{\mathcal{C}^{\alpha}(\Omega)} \leq \|f\|_{\mathcal{C}^{\alpha}(\Omega)}\| \nu \|_{(\mathcal{C}^{\alpha}(\Omega))'}. $$
As a consequence,   $f\mapsto \mathbb{A}_\C(f, \nu)$ and   $\nu\mapsto \mathbb{A}_\C(f, \nu)$ are  bounded operators.  
\end{Lemma}

{\it Step 1:  Fr\'echet diferentiability of $\Psi$.}
The following result yields the Fr\'echet diferentiability of $\Psi$ at $ (f_*, g_*) $, which  implies that 
\begin{equation}\label{limitFrechet}
     \frac{\| {\Psi(f_{n,m}, g_{n,m})-\Psi(f_*, g_*)}-  D_{(f_*, g_*)}\Psi_{\mP,\mQ}(\delta_{n,m}) \|_{\Co}}{\| \delta_{n,m}\|_{\CoAlpha}}\xrightarrow{\P}0.
\end{equation}
for $\delta_{n,m}= (f_{n,m}-f_*, g_{n,m}-g_*)$.
\begin{Lemma}\label{Lemma:Frechet}
The functional $\Psi:\CoAlpha\to \CoAlpha $ is  Fr\'echet differentiable at $(f_*,g_{*})$ with derivative 
$$  \left(\begin{array}{c}
    h_1  \\
       h_2 
   \end{array}\right) \mapsto \left(\begin{array}{c}
    h_1  \\
       h_2 
   \end{array}\right) +\left(\begin{array}{c}
      \AQ h_2    \\
        \bAP h_1 
   \end{array}\right).$$
\end{Lemma}
\begin{proof}
Note that it is enough to check out that the limits
\begin{equation*}
\limsup_{\|h_1\|_{\alpha}\rightarrow 0}\frac{\|\log\left( \mathbb{A}_\C(e^{f_{*}+h_1}, \mP)\right)-\log\left( \mathbb{A}_\C(e^{f_{*}}, \mP)\right)-\AP(h_1)\|_{\alpha}}{\|h_1\|_{\alpha}}
\end{equation*}
and 
\begin{equation*}
\limsup_{\|h_2\|_{\alpha}\rightarrow 0}\frac{\|\log\left( \mathbb{A}_\C(e^{g_{*}+h_2}, \mQ)\right)-\log\left( \mathbb{A}_\C(e^{g_{*}}, \mQ)\right)-\AQ(h_2)\|_{\alpha}}{\|h_1\|_{\alpha}}
\end{equation*}
are $0$. We focus only on the first one,  as  the second follows by the same arguments.
Lemma~\ref{Lemma:Hadamard} yields
$\left\| e^{f^*+h_1}- e^{f^*}-h_1 e^{f^*}\right\|_{\alpha}=o(\|h_1\|_\alpha).$
The linearity and continuity of the operator 
$f\to \mathbb{A}_\C(f, \mP)$ implies its Fr\'echet differentiability. By composition, the operator $f\to \mathbb{A}_\C(e^f, \mP)$ is also Fr\'echet differentiable. Since $ \mathbb{A}_\C(e^f, \mP) >0$, Lemma~\ref{Lemma:Hadamard} applied to the logarithm concludes the proof.
\end{proof}

{\it Step 2: Invertibility of $\CoAlpha$. }
We adopt matrix notation, i.e.,
\begin{multline*}
    \left(\begin{array}{cc}
  (1-\AQ\bAP)^{-1} & -(1-\AQ\bAP)^{-1}\AQ  \\
      - \bAP (1-\AQ\bAP)^{-1} & (1-\bAP\AQ)^{-1} 
   \end{array}\right)\left(\begin{array}{c}
    f  \\
      g
   \end{array}\right)\\
   =\left(\begin{array}{c}
     (1-\AP\bAP)^{-1}f -(1-\AP\bAP)^{-1}\AP g  \\
      - \bAP (1-\AP\bAP)^{-1}f+ (1-\bAP\AQ)^{-1} g
   \end{array}\right).
\end{multline*}
\begin{Lemma}\label{Lemma:equivalent}
Let $\Omega\subset{\R^d}$ be a compact set, ${\rm P}, {\rm Q}\in \mathcal{P}(\Omega)$, then 
\begin{enumerate}[(i)]
   \item \label{one}  $(1-\AQ\bAP)$  and $(1-\bAP\AQ)$ are continuously invertible operators in $\mathcal{C}^{\alpha}(\Omega)$ and $\mathcal{C}^{\alpha}_0(\Omega)$ respectively. 
    \item  \label{two} $\AP=\bAP$ in the space
    $ \{ f\in \mathcal{C}^{\alpha}(\Omega):\quad \int f(\x)d\mP(\x)=0 \}. $
    \item \label{three}the relation 
    \begin{equation*}
    \left(\begin{array}{cc}
  (1-\AQ\bAP)^{-1} & -(1-\AQ\bAP)^{-1}\AQ  \\
      - \bAP (1-\AQ\bAP)^{-1} & (1-\bAP\AQ)^{-1} 
   \end{array}\right)=(D_{(f_*, g_*)}\Psi)^{-1}
\end{equation*}
holds in $\CoAlpha$. 
\end{enumerate}
Moreover, \eqref{one}, \eqref{two} and \eqref{three} also hold in $L^2(\mP)\times L^2_0(\mQ) $ instead of $\CoAlpha$. 
\end{Lemma}
\begin{proof}
If we show that the operators $\AQ$  $\bAP$ are compact and that $(1-\bAP\AQ)$ and $(1-\AQ\bAP)$ are injective,  the Fredholm alternative  \cite[Theorem 6.6]{Brezis} and continuous inverse theorem  \cite[Corollary 2.7]{Brezis} would prove \eqref{one}.  Claim  \eqref{two} follows by \eqref{proofodii} below and the last claim by basic algebra. The proof in 
$L^2(\mP)\times L^2_0(\mQ) $
  is the same; to avoid repeating arguments, we leave the details to the reader.

{\it Compactness of $\AP$ and $\AQ$:}
Let $\{f_k\}_{k\in \N}$ be a bounded sequence in $\mathcal{C}^{\alpha}(\Omega)$. Since the sequence 
\begin{equation}
    \label{boundedDer}
    \AP f_k=\int \xi_{\mP,\mQ}(\x,\cdot)f_k(\x)d\mP(\x)\in \mathcal{C}^{\alpha+1}(\Omega), \quad  k\in N
\end{equation}
has its derivatives up to order $\alpha+1$  uniformly bounded by some $C(\Omega, \alpha, d)$ \cite[Lemma~4.3]{del2022improved},   the Arzelà--Ascoli  theorem yields the relatively compactness of $\{\AP f_k\}_{k\in \N}$ in $\mathcal{C}^{\alpha}(\Omega)$. The same argument applies to  $\AQ$.

{\it  Injectivity of Fredholm operators:}
We prove it by \textit{reductio ad absurdum}. Suppose that $\AQ\bAP f=\AQ(\AP-\mQ \AP)f=\lambda f$, for some  nonzero  $f\in \mathcal{C}^{\alpha}({\Omega})$ and $\lambda\in \{-1, 1\} $. In this case, since the optimality condition \eqref{optimallityCodt01} implies
\begin{equation}
    \label{proofodii}
  \mP \AQ \AP f=   \mQ \AP f=\iint \xi_{\mP,\mQ}(\x,\y)f(\x)d\mP(\x)d\mQ(\y)= \lambda \int fd\mP,
\end{equation}
 we have
 \begin{multline*}
  \mP  \AQ\bAP f= \iint\xi_{\mP,\mQ}(\x,\y)((\AP-\mQ \AP)f)(\y) d\mQ(\y)d\mP(\x)=\lambda \mQ(\AP-\mQ \AP)f=0, 
 \end{multline*}
 and  hence  by assumption $ 0=\int f d\mP  = \mQ \AP f $.  We therefore obtain 
\begin{align*}
 f(\x')^2= (\AQ \bAP f(\x') )^2= \ER (  \AQ \ER \AP f(\x') )^2= \left( \int \xi_{\mP,\mQ}(\x',\y)\int \xi_{\mP,\mQ}(\x,\y)f(\x)d\mP(\x)d\mQ(\y)\right)^2
\end{align*}
 for all $\x'\in \Omega.$ On the other hand, Jensen's inequality gives
\begin{align*}
 f(\x')^2  \leq  \int \xi_{\mP,\mQ}(\x',\y)\left( \int \xi_{\mP,\mQ}(\x,\y)f(\x) d\mP(\x)\right)^2 d\mQ(\y)\leq  \|f\|_{\infty}^2, \quad \text{for all $\x'\in \Omega$},
\end{align*}
and the first inequality is strict unless 
\begin{equation}
    \label{finalabs}
    \text{$\int \xi_{\mP,\mQ}(\x,\y)f(\x)\mP(\x)=c\in \R$  for   $\mQ$-a.e. $\y$.}
\end{equation}
 Since there exists $\x'\in \Omega$  such that $  \mid f(\x') \mid=\|f\|_{\infty}$, \eqref{finalabs} holds. Therefore,  $$ \bar{f}_{\mP,\mQ}: \x\mapsto\log\left(\frac{e^{f_{\mP,\mQ}(\x)}(f(\x)+1+\|f\|_{\infty})}{c+1+\|f\|_{\infty}}\right)$$ 
is a solution of the dual problem \eqref{dual_entrop}.  Since it is unique, we obtain $\frac{e^{f_{\mP,\mQ}(\x)}(f(\x)+1+\|f\|_{\infty})}{c+1+\|f\|_{\infty}}=e^{f_{\mP,\mQ}(\x)}$ and $f=c$ in $\Omega$. Since $f$ is centered, we conclude $f=0$.
\end{proof} 
{\it Step 3: Linearization of $ \Psi_{n,m}- \Psi $.} The optimality condition $\Psi_{n,m}\left(\begin{array}{c}
     f_{n,m} \\
   g_{n,m}
\end{array}\right)=\Psi\left(\begin{array}{c}
     f_* \\
g_*
\end{array}\right)=0$ and \eqref{limitFrechet} yields 
\begin{equation}\label{toeachange}
\Psi\left(\begin{array}{c}
     f_{n,m} \\
    g_{n,m} 
\end{array}\right)-\Psi_{n,m}\left(\begin{array}{c}
     f_{n,m} \\
    g_{n,m} 
\end{array}\right)= D_{(f_*, g_*)}\Psi_{\mP,\mQ}(\delta_{n,m}) +o_{\P}\left(\sqrt{ \frac{n+m}{n\,m}}\right).
\end{equation}
The following result allows to exchange, up to additive $o_{\P}\left(\sqrt{ \frac{n+m}{n\,m}}\right)$ terms, $(f_{n,m},g_{n,m})$ by $(f_{*},g_{*})$ in \eqref{toeachange}.
\begin{Lemma}\label{Lemma:Linearization}
The asymptotic equality
\begin{equation*}
\Psi\left(\begin{array}{c}
     f_{n,m} \\
    g_{n,m} 
\end{array}\right)-\Psi_{ n, m }\left(\begin{array}{c}
      f_{n, m}  \\
 	  g_{n, m} 
\end{array}\right)=\Psi\left(\begin{array}{c}
     f_{*} \\
   g_{*} 
\end{array}\right)-\Psi_{n,m}\left(\begin{array}{c}
     f_{*} \\
   g_{*} 
\end{array}\right)+o_{\P}\left(\sqrt{ \frac{n+m}{n\,m}}\right)
\end{equation*}
holds in $\CoAlpha$. As a consequence, 
\begin{equation*}
   \Psi\left(\begin{array}{c}
     f_{*} \\
    g_{*} 
\end{array}\right)-\Psi_{n,m}\left(\begin{array}{c}
     f_{*} \\
    g_{*} 
\end{array}\right)= D_{(f_*, g_*)}\Psi_{\mP,\mQ}(\delta_{n,m}) +o_{\P}\left(\sqrt{ \frac{n+m}{n\,m}}\right)
\end{equation*}
holds in $\CoAlpha$. 
\end{Lemma} 
\begin{proof} 

It is enough to show that 
\begin{equation}\label{auxiliaryLemmafirstCoord}
    \log\left(\mathbb{A}_\C(e^{f_{n,m}}, \mP)\right)-\log\left(\mathbb{A}_\C(e^{f_{n,m}}, \mP_n)\right)
    -\log\left(\mathbb{A}_\C(e^{f_{*}}, \mP)\right)+\log\left(\mathbb{A}_\C(e^{f_{*}}, \mP_n)\right)
\end{equation}
and 
$$ 
 \log\left(\mathbb{A}_\C(e^{g_{n,m}}, \mQ)\right)-\log\left(\mathbb{A}_\C(e^{g_{n,m}}, \mQ_m)\right)
    -\log\left(\mathbb{A}_\C(e^{g_{*}}, \mQ)\right)+\log\left(\mathbb{A}_\C(e^{g_{*}}, \mQ_m)\right)
$$
are both $o_{\P}\left( \sqrt{\frac{n+ m}{n\, m}}\right)$ in $\mathcal{C}^{\alpha}(\Omega)$. To avoid repeated arguments we only prove this claim for \eqref{auxiliaryLemmafirstCoord}. 
 In accordance with Lemma~\ref{Lemma:Hadamard}---applied to the natural logarithm---and equations \eqref{Lamma_exponentialEquationOperator1}, \eqref{Lamma_exponentialEquationOperator2} and \eqref{Lamma_exponentialEquationOperator3}   we obtain the estimates
 \begin{align*}
    \log\left(\mathbb{A}_\C(e^{f_{n,m}}, \mP)\right)&=   \log\left(\mathbb{A}_\C(e^{f_{*}}, \mP)\right)+\frac{ \mathbb{A}_\C(e^{f_{n,m}}, \mP)}{ \mathbb{A}_\C(e^{f_{*}}, \mP)}-1
    +o_{\P}\left(\sqrt{\frac{n+m}{n\,m}}\right),\\
     \log\left(\mathbb{A}_\C(e^{f_{n,m}}, \mP_n)\right)&=\log\left(\mathbb{A}_\C(e^{f_{*}}, \mP)\right)+\frac{\mathbb{A}_\C(e^{f_{n,m}}, \mP_n)}{\mathbb{A}_\C(e^{f_{*}}, \mP)}-1+o_{\P}\left(\sqrt{\frac{n+m}{n\,m}}\right),
     \\
     \log\left(\mathbb{A}_\C(e^{f_{*}}, \mP_n)\right)&=\log\left(\mathbb{A}_\C(e^{f_{*}}, \mP)\right)+\frac{\mathbb{A}_\C(e^{f_{*}}, \mP_n)}{\mathbb{A}_\C(e^{f_{*}}, \mP)}-1+o_{\P}\left(\sqrt{\frac{n+m}{n\,m}}\right).
 \end{align*}
From here we derive the the following rewriting of \eqref{auxiliaryLemmafirstCoord}; 
\begin{equation}\label{ProofLermma24relation4}
  \frac{\mathbb{A}_\C(e^{f_{n,m}}-e^{f_*}, \mP_n-\mP)}{\mathbb{A}_\C(e^{f_{*}}, \mP)}+o_{\P}\left(\sqrt{\frac{n+m}{n\,m}}\right).
\end{equation}
Then, since $\mathbb{A}_\C(e^{f_{*}}, \mP)$ is bounded away from $0$ and infinity, by Lemma~\ref{BanachAlgebra} the proof is concluded by showing 
$
{\mathbb{A}_\C(e^{f_{n,m}}-e^{f_*}, \mP_n-\mP)}=o_{\P}\left(\sqrt{\frac{n+m}{n\,m}}\right).
$ That holds due to Lemma~\ref{Lemma:BoundOnA}, \eqref{Inequqlitt2} and the fact that the unit ball of $\mathcal{C}^{\alpha}(\Omega)$ is uniformly Donsker. 

\end{proof}
With the results we have obtained, we are now in a position to complete the proof of the theorem. For $\alpha\geq s$, it holds that 
$ \sqrt{ \frac{n\,m}{n+m}}\| \mP_n-\mP  \|_{(\mathcal{C}^{\alpha}(\Omega))'} $ and $ \sqrt{ \frac{n\,m}{n+m}}\| \mQ_m-\mQ  \|_{(\mathcal{C}^{\alpha}(\Omega))'} $ are $ \mathcal{O}_\P(1)$.   Hence Lemma~\ref{Lemma:BoundOnA} and Lemma~\ref{Lemma:Hadamard} (applied to the natural logarithm) yield
\begin{align*}
     &\Psi\left(\begin{array}{c}
     f_{*} \\
   g_{*} 
\end{array}\right)-\Psi_{n,m}\left(\begin{array}{c}
     f_{*} \\
   g_{*} 
\end{array}\right)\\
&=\left(\begin{array}{c}
    \log\left(\mathbb{A}_\C(e^{g_*}, \mQ) \right) -\log\left(\mathbb{A}_\C(e^{g_*}, \mQ_m) \right) 
    \\
       \log\left(\mathbb{A}_\C(e^{f_*}, \mP) \right)-\log\left(\mathbb{A}_\C(e^{f_*}, \mP_n) -\mQ(\log\left(\mathbb{A}_\C(e^{f_*}, \mP) \right)-\log\left(\mathbb{A}_\C(e^{f_*}, \mP_n) \right))\right)
    \end{array}\right)
    \\&=\left(\begin{array}{c}
   e^{f_*} \mathbb{A}_\C(e^{g_*}, \mQ-\mQ_m)  
    \\
       e^{g_*} \mathbb{A}_\C(e^{f_*}, \mP-\mP_n) 
    \end{array}\right)+o_\P\left( \sqrt{\frac{n+m}{n\, m} }\right)\quad {\rm in}\  \mathcal{C}^{\alpha}(\Omega)\times \mathcal{C}^{\alpha}_0(\Omega).
\end{align*}
 Lemmas~\ref{Lemma:equivalent} and~\ref{Lemma:Linearization}  give 
$$  \left(\begin{array}{c}
    f_{n,m}- f_{*} \\
  g_{n,m}- g_{*}
\end{array}\right) =(D_{(f_*, g_*)}\Psi)^{-1}\left(\begin{array}{c}
   e^{f_*} \mathbb{A}_\C(e^{g_*}, \mQ-\mQ_m)  
    \\
       e^{g_*} \mathbb{A}_\C(e^{f_*}, \mP-\mP_n) 
    \end{array}\right)+o_\P\left( \sqrt{\frac{n+m}{n\, m} }\right)\quad {\rm in}\  \mathcal{C}^{\alpha}(\Omega)\times \mathcal{C}^{\alpha}_0(\Omega). $$
As a consequence,
\begin{multline}
    \label{beforeexchanging}
     \left(\begin{array}{c}
        f_{\mP_n,\mQ_m}- f_{\mP,\mQ}  \\
        g_{\mP_n,\mQ_m}- g_{\mP,\mQ}  
    \end{array}\right)
    = \epsilon \left(\begin{array}{c}
           (1-\AQ\bAP)^{-1}\AQ {\bf i}_{\xi_{\mP, \mQ}}  (\mP_n-\mP)-(1-\AQ\bAP)^{-1} {\bf i}_{\xi_{\mQ, \mP} } (\mQ_m-\mQ) \\
    \bAP (1-\AQ\bAP)^{-1}  {\bf i}_{\xi_{\mQ, \mP} } (\mQ_m-\mQ)-(1-\bAP\AQ)^{-1} {\bf i}_{\xi_{\mP, \mQ}}  (\mP_n-\mP) \end{array}\right)\\+o_{\P}\left(\sqrt{\frac{n+ m }{n\,m}}\right). 
\end{multline}
Since 
$\int {\bf i}_{\xi_{\mP, \mQ}}  (\mP_n-\mP) d\mQ= \frac{1}{n} \sum_{k=1}^n \int\xi_{\mP,\mQ}(\X_k,\y)-\E\left( \xi_{\mP,\mQ}(\X, \y)\right)d\mQ(\y)=0,$
in view of Lemma~\ref{Lemma:equivalent}~\eqref{two}, we can thus exchange $\bAP$ by $\AP$ in \eqref{beforeexchanging} and the proof of the theorem is completed.
\end{proof}

\section{Central limit theorem for the solution of the primal problem and Sinkhorn distances}\label{sec:TCLcouplings}
This section covers the weak limit of the quantity
\begin{equation}
    \label{thequqntity}
      \sqrt{\frac{n\, m}{n+m}}\int \eta\, (d\pi_{\mP_n,\mQ_m} -d\pi_{\mP,\mQ} ), \quad \text{where $\eta\in L^\infty(\mP\otimes\mQ)$.}
\end{equation}
 %he organization of the section is as follows; first we introduce notations to simplify the description of the limit that allow us to formulate the main result Theorem~\ref{Theorem:TCLcouplings}, which describes the first order linearization, and  its main consequence Corollary~\ref{Corollary:TCLcouplings}. To conclude we observe that Corollary~\ref{Corollary:TCLcouplings}, at the same time, implies Corollary~\ref{Corollary:TCLcouplingsDiv}, which gives a description of the limit of the Sinkhorn distance introduced in \cite{Cut13}. The section concludes with the proof of Theorem~\ref{Theorem:TCLcouplings}.

Before stating this result, for a fixed function $\eta\in L^\infty(\mP\otimes \mQ)$ we introduce the notation;
\begin{align*}
    &\eta_{\x} :\x\to \int \eta(\x,\y)\xi_{\mP,\mQ} (\x,\y)d\mQ(\y)\ \ \text{and}\ \  
     &\eta_{\y} :\y\to \int \eta(\x,\y)\xi_{\mP,\mQ} (\x,\y)d\mP(\x),
\end{align*}
which substantially simplifies the description of the first-order decomposition of \eqref{thequqntity}, described in the following theorem.
\begin{Theorem}\label{Theorem:TCLcouplings}
Let $\Omega\subset{\R^d}$ be a compact set, ${\rm P}, {\rm Q}\in \mathcal{P}(\Omega)$ and ${\rm P}_n$ (resp. $\mQ_m$) be the empirical measure of the i.i.d. sample $\X_1,\dots,\X_n$ (resp. $\Y_1,\dots,\Y_m$) distributed as $\rm P$ (resp. $\mQ$). Then, if $m=m(n)\to \infty$, $\frac{m}{n+m}\to \lambda\in (0,1)$ and $\eta\in L^\infty(\mP\otimes \mQ)$,
\begin{multline*}
    \int \eta\, (d\pi_{\mP_n,\mQ_m} -d\pi_{\mP,\mQ} )\\
    ={  {\frac{1}{{n}}}\sum_{k=1}^n (1-\AQ\AP)^{-1}\big(\eta_{\x} -\AQ\eta_{\y} \big)(\X_k)+   { \frac{1}{{m}}}{\sum_{j=1}^m (1-\AP\AQ)^{-1}\big(\eta_{\y} -\AP\eta_{\x} \big)(\Y_j)}}\\
    +o_{\P}\left(\sqrt{\frac{n+ m}{n\,m}}\right).
\end{multline*}
\end{Theorem}
Theorem~\ref{Theorem:TCLcouplings} gives the first-order decomposition of the solutions of the  regularized optimal transport problem \ref{kanto_entrop}. We recall that, for a regularization based on the Schr\"{o}dinger bridge, \cite{ChaosDecom} arrived at exactly the same thing, conjecturing the truth of Theorem~\ref{Theorem:TCLcouplings}. The techniques  of the proofs are completely different; ours is based on the theory of empirical processes, while that of \cite{ChaosDecom} on a change of measure and projections in $L^\infty(\mP\otimes\mQ)$.

%Since $ \mP\AQ(\eta_{\y}) = \mP (\eta_{\y}) $ and $ \mP\AP(\eta_{\x})=\mP(\eta_{\x}) $, we have 
%$ \mP \big(\eta_{\x} -\AQ\eta_{\y} \big)=\mQ \big(\eta_{\y} -\AP\eta_{\x} \big)=0.$ Set $k\in \N$, Lemma~\ref{Lemma:equivalent}~\eqref{two} implies that $$\mP(\AQ\AP)^k\big(\eta_{\x} -\AQ\eta_{\y} \big)=\mP(\AP\AQ)^k\big(\eta_{\y} -\AP\eta_{\x} \big)=0$$
%and therefore
%\begin{align*}
 %    \mP(1-\AQ\AP)^{-1}\big(\eta_{\x} -\AQ\eta_{\y} \big)+    \mQ(1-\AP\AQ)^{-1}\big(\eta_{\y} -\AP\eta_{\x} \big)=0.
%\end{align*} 
%Then the first-order decomposition of  Theorem~\ref{Theorem:TCLcouplings} is centered and we obtain  as an immediate consequence the limit of \eqref{thequqntity}.

\begin{Corollary}\label{Corollary:TCLcouplings}
Let $\Omega\subset{\R^d}$ be a compact set, ${\rm P}, {\rm Q}\in \mathcal{P}(\Omega)$ and ${\rm P}_n$ (resp. $\mQ_m$) be the empirical measure of the i.i.d. sample $\X_1,\dots,\X_n$ (resp. $\Y_1,\dots,\Y_m$) distributed as $\rm P$ (resp. $\mQ$). Then, if $m=m(n)\to \infty$, $\frac{m}{n+m}\to \lambda\in (0,1)$ and $\eta\in L^\infty(\mP\otimes \mQ)$,
\begin{equation*}
    \sqrt{\frac{n\, m}{n+m}}\left(\int \eta d \pi_{\mP_n,\mQ_m} -\int \eta d \pi_{\mP,\mQ}  \right)\xrightarrow{w} N(0,\sigma^2_{\lambda, \epsilon}(\eta)),
\end{equation*}
where the variance $\sigma^2_{\lambda,\epsilon}(\eta)$ is 
\begin{multline*}
    \lambda \operatorname{Var}_{\X\sim \mP}\left[ (1-\AQ\AP)^{-1}\big(\eta_{\x} -\AQ\eta_{\y} \big)(\X)\right]\\+ (1-\lambda )\operatorname{Var}_{\Y\sim \mQ}\left[ (1-\AP\AQ)^{-1}\big(\eta_{\y} -\AP\eta_{\x} \big)(\Y)\right].
\end{multline*}
Moreover, in the one-sample case we have
\begin{equation*}
    \sqrt{n}\left(\int \eta d \pi_{\mP_n,\mQ} -\int \eta d \pi_{\mP_n,\mQ}  \right)\xrightarrow{w}  N(0,\sigma^2_{\mP,  \epsilon }(\eta)),
\end{equation*}
with 
$  \sigma^2_{\mP,  \epsilon }(\eta)=\operatorname{Var}_{\X\sim \mP}\left[ (1-\AQ\AP)^{-1}\big(\eta_{\x} -\AQ\eta_{\y} \big)(\X)\right].$
\end{Corollary}
%Note that, in the proof of Theorem~\ref{Theorem:TCLcouplings}  two terms are random variables: the difference between densities---which depends on the potentials---and the difference between  the empirical processes. Each term is treated separately, giving rise to the linearization described in Theorem~\ref{Theorem:TCLcouplings}.\\
\subsection{Applications}
An immediate corollary of Theorem~\ref{Theorem:TCLcouplings}  is its application to the square  norm function, where we obtain the weak limit of the Sinkhorn cost introduced in \cite{Cut13}. Formally it is defined as 
\begin{equation*}
    d_S (\mP,\mQ)=\E_{(\X,\Y)\sim \pi_{\mP,\mQ} }\left[\frac{\|\X-\Y\|^2}{2}\right],
\end{equation*}
and represents the cost of `transporting mass' from $\mP$ to $\mQ$ when using the coupling given by the entropic regularization. 
\begin{Corollary}\label{Corollary:TCLcouplingsDiv}
Let $\Omega\subset{\R^d}$ be a compact set, ${\rm P}, {\rm Q}\in \mathcal{P}(\Omega)$ and ${\rm P}_n$ (resp. $\mQ_m$) be the empirical measure of the i.i.d. sample $\X_1,\dots,\X_n$ (resp. $\Y_1,\dots,\Y_m$) distributed as $\rm P$ (resp. $\mQ$). Then, if $m=m(n)\to \infty$, $\frac{m}{n+m}\to \lambda\in (0,1)$,
\begin{equation*}
    \sqrt{\frac{n\, m}{n+m}}\left( d_S (\mP_n,\mQ_m)- d_S (\mP,\mQ) \right)\longrightarrow N\left(0,\sigma^2_{\lambda, \epsilon}\left(\frac{\|\cdot-\cdot \|^2}{2}\right)\right),\ \ \text{weakly},
\end{equation*}
where the variance $\sigma^2_{\lambda,\epsilon}(\frac{\|\cdot-\cdot \|^2}{2})$ is defined in Corollary~\ref{Corollary:TCLcouplings}  for the function $(\x,\y)\to \frac{\|\x-\y \|^2}{2}$.
Moreover, in the one-sample case we have
\begin{equation*}
    \sqrt{n}\left( d_S (\mP_n,\mQ)- d_S (\mP,\mQ) \right)\longrightarrow N\left(0,\sigma^2_{\mP, \epsilon}\left(\frac{\|\cdot-\cdot \|^2}{2}\right)\right),\ \ \text{weakly},
\end{equation*}
with $\sigma^2_{\mP, \epsilon}\left(\frac{\|\cdot-\cdot \|^2}{2}\right)$ as in Corollary~\ref{Corollary:TCLcouplings} .
\end{Corollary}
Another interesting application of Corollary~\ref{Corollary:TCLcouplings} is to the function $(\x,\y)\mapsto \mathbbm{1}_{\|\x-\y\|^2\leq t}$, for $t\geq 0$. \cite{KlaTamMun20} computed the regularized optimal transport problem to match two
protein intensity distributions and defined the regularized
colocalization measure RCol
\begin{equation*}
     \operatorname{RCol}(\pi_{\mP,\mQ} ,t)=\int \mathbbm{1}_{\|\x-\y\|^2\leq t} d \pi_{\mP,\mQ} (\x, \y)=\pi_{\mP,\mQ} ( \|\cdot-\cdot\|^2\leq t),
\end{equation*}
which represents the mass of the pixel intensity 
 transported on scales smaller or equal to $t$. Theorem 7.1. in \cite{KlaTamMun20} gives confidence intervals for the discretized images (finite number of pixels). The following result extend it to general probability distributions representing the pixels.
\begin{Corollary}\label{Corollary:TCLcouplingsKlatt}
Let $\Omega\subset{\R^d}$ be a compact set, ${\rm P}, {\rm Q}\in \mathcal{P}(\Omega)$ and ${\rm P}_n$ (resp. $\mQ_m$) be the empirical measure of the i.i.d. sample $\X_1,\dots,\X_n$ (resp. $\Y_1,\dots,\Y_m$) distributed as $\rm P$ (resp. $\mQ$). Then, if $m=m(n)\to \infty$, $\frac{m}{n+m}\to \lambda\in (0,1)$,
\begin{equation*}
    \sqrt{\frac{n\, m}{n+m}}\left(  \operatorname{RCol}(\pi_{\mP_n,\mQ_m} ,t)-  \operatorname{RCol}(\pi_{\mP,\mQ} ,t)\right)\longrightarrow N\left(0,\sigma^2_{\lambda, \epsilon}\left( \mathbbm{1}_{\|\cdot-\cdot\|^2\leq t}\right)\right),\ \ \text{weakly},
\end{equation*}
where the variance $\sigma^2_{\lambda,\epsilon}(\mathbbm{1}_{\|\cdot-\cdot\|^2\leq t})$ is defined in Corollary~\ref{Corollary:TCLcouplings}  for the function $(\x,\y)\mapsto \mathbbm{1}_{\|\x-\y\|^2\leq t}$.
Moreover, in the one-sample case we have
\begin{equation*}
    \sqrt{n}\left(  \operatorname{RCol}(\pi_{\mP_n,\mQ} ,t)-  \operatorname{RCol}(\pi_{\mP,\mQ} ,t) \right)\longrightarrow N\left(0,\sigma^2_{\mP, \epsilon}\left(\mathbbm{1}_{\|\cdot-\cdot\|^2\leq t}\right)\right),\ \ \text{weakly},
\end{equation*}
with $\sigma^2_{\mP, \epsilon}\left(\mathbbm{1}_{\|\cdot-\cdot\|^2\leq t}\right)$ as in Corollary~\ref{Corollary:TCLcouplings} 
\end{Corollary}

\subsection{Proof of Theorem~\ref{Theorem:TCLcouplings}}

Recall that $1-\AQ\AP$ and $1-\AP\AQ$ are self adjoint due to 
the following relation, which is consequence of Fubini's theorem,
\begin{equation}
    \label{conjugate}
   \langle  f,  \AQ g\rangle_{L^2(\mP)}=\langle  \AP f, g \rangle_{L^2( \mQ )}, \quad \text{for all $f\in L^2(\mP), g\in L^2(\mQ)$}.
\end{equation}

 %As has already been done before, we assume that $\epsilon=1$ and  avoid any super/sub-index related with $\epsilon$.
 For the sake of readability, we adopt the notation \eqref{notationForProofs}.
We define first the 
bilinear form $ \mathbb{B}_{\C} $ as 
\begin{align*}
    \mathbb{B}_{\C}: \mathcal{C}^{1}(\Omega^2)  \times  (\mathcal{C}^{1}(\Omega^2))'  &\to \R \\
    (h, \nu)& \mapsto   \nu( h  \C ), 
\end{align*} 
which for measures $\nu\in (\mathcal{C}(\Omega))'$ takes the form 
$ \mathbb{B}_{\C}(h, \nu)=\int h(\x, \y) \C(\x, \y) d\nu(\x, \y). $ The following are the main properties of the form $\mathbb{B}_{\C}$. The proofs of \eqref{BoundBoperator} and \eqref{BoundOnfixedhB} are straightforward, while the proof of \eqref{BoundsOnEmpiricalNu} is direct consequence of \cite[Theorems 2.7.11 and 2.7.1 ]{Vart_Well}. 
\begin{Lemma}\label{Lemma:propB}
The following holds: 
\begin{itemize}
    \item (Continuity) For all $(h, \nu)\in \mathcal{C}^{1}(\Omega^2)  \times  (\mathcal{C}^{1}(\Omega^2))'$
      \begin{equation}
    \label{BoundBoperator}
    | \mathbb{B}_{\C}(h, \nu)|\leq \|h\C\|_{(\mathcal{C}^{1}(\Omega^2))} \|\nu\|_{(\mathcal{C}^{1}(\Omega^2))'}\leq C\|h\|_{(\mathcal{C}^{1}(\Omega^2))} \|\nu\|_{(\mathcal{C}^{1}(\Omega^2))'}.
\end{equation}
\item (Consistent empirical process) 
\begin{equation}
\label{BoundsOnEmpiricalNu}
\|\eta (\mQ_m\otimes (\mP_n-\mP))\|_{(\mathcal{C}^1(\Omega))'}, \  \|\eta ((\mQ_m-\mQ)\otimes \mP_n)\|_{(\mathcal{C}^1(\Omega))'}=o_\P(1).
\end{equation}
\item (Rate for fixed $ h\in \mathcal{C}^{1}(\Omega^2) $) 
\begin{equation}
    \label{BoundOnfixedhB}
    |\mathbb{B}_{\C}(h, \eta\, (\mQ_m-\mQ)\otimes (\mP_n-\mP)))|=\mathcal{O}_\P\left(\frac{n+ m}{n\, m}\right).
\end{equation}
\end{itemize}
\end{Lemma}
We also have the following ``rules of calculus'' of the operators $\AP$, $\AQ$ and ${\bf i}_{\xi_{\mP, \mQ}} $ in the space  
\begin{equation}
    \label{OPeratorS}
    \mathcal{S}=\overline{\{\nu\in (\mathcal{C}(\Omega))': \quad \nu(\Omega)=0\}}^{(\Cs)'}\subset (\Cs)'.
\end{equation}
Such a Banach space is convenient due to two facts; (i) the equality 
$$  \nu \AP^2 {\bf i}_{\xi_{\mP,\mP}}(\mu) = \langle   \AP {\bf i}_{\xi_{\mP,\mP}}(\mu),  {\bf i}_{\xi_{\mP,\mP}}(\nu) \rangle_{L^2(\mP)} $$
holds for $\mu, \nu\in \mathcal{S}$ due to Fubinni's theorem; (ii)
$$ \mQ \, {\bf i}_{\xi_{\mP,\mQ}}(\mu)=\int \xi_{\mP, \mQ}(\x, \y) d\mQ(\y) d \mu(\x)=0 , $$
so that $(1-\AP\AQ)^{-1}$ and $(1-\AQ\AP)^{-1}$ are well-defined in $ \mathcal{S}$; 
and the $\mP$-Gaussian bridge belongs to $\mathcal{S} $. 
 
\begin{Lemma}[Rules of Calculus] Set $ \nu\in  \mathcal{S}  $  and $f\in \Cs $. 
    It holds that
    \begin{align}
   &\AP(1-\AQ\AP)^{-1}=  (1-\AP\AQ)^{-1}\AP \label{ruleApAq1},  && \AQ(1-\AP\AQ)^{-1}=  (1-\AQ\AP)^{-1}\AQ ,\\
   &\langle f,{\bf i}_{\xi_{\mQ, \mP}}\nu\rangle_{L^2(\mP)}=\nu( \AP f),\quad {\rm and} &&\langle f,{\bf i}_{\xi_{\mP, \mQ}} \nu\rangle_{L^2(\mQ)}=\nu( \AQ f). \label{ruleApAq2}
\end{align}
\end{Lemma}
\begin{proof}
 \eqref{ruleApAq1} holds  due to
 \begin{align*}
      \AP&=(1-\AP\AQ)^{-1}(1-\AP\AQ)\AP\\
      &=(1-\AP\AQ)^{-1}(\AP-\AP\AQ\AP)\\
    &=(1-\AP\AQ)^{-1}\AP(1-\AQ\AP).
 \end{align*}
We prove  \eqref{ruleApAq2} for $ \nu\in  {(\mathcal{C}(\Omega))'} $. Fubini's theorem implies that 
\begin{align*}
    \langle f,{\bf i}_{\mQ, \mP}\nu\rangle_{L^2(\mP)}&=\int f(\x) \xi_{\mP, \mQ}(\x, \y) d(\mP\otimes \nu)(\x, \y)\\
    &=\int \left(\int f(\x) \xi_{\mP, \mQ} (\x, \y) d\mP(\x) \right) d\nu(\y),
\end{align*}
yielding the result. In the  general case $ \nu\in  \mathcal{S}  $ the same result holds due to the continuity of the operators and a standard density argument. 
\end{proof}
We call $ v_{n,m}^\mQ= \sqrt{\frac{n\, m}{n+m}} (\mQ_m-\mQ) $ and $ v_{n,m}^\mP= \sqrt{\frac{n\, m}{n+m}} (\mP_n-\mP) $.  
With this estimates in mind, 
we  split \eqref{thequqntity} in three terms: 
  $\int \eta\, (d\pi_{\mP_n,\mQ_m} -d\pi_{\mP,\mQ} )=A+B+C$, where each of them can be easily estimated via \eqref{BoundOnfixedhB}, \eqref{BoundsOnEmpiricalNu}, \eqref{BoundBoperator} and Lemma~\ref{Lemmainequqlity} as follows. First the term capturing the difference between the densities   
 \begin{align*}
     A&= \sqrt{\frac{n\, m}{n+m}}\int \eta(\x,\y)  \mathbb{C}(\x, \y) \left(e^{h_{n,m}(\x,\y)}-e^{h_*(\x,\y)} \right)d \mQ_m(\y) d\mP_n(\x)\\
     &=\sqrt{\frac{n\, m}{n+m}} \mathbb{B}_{\C}  \left(e^{h_{n,m}}-e^{h_*}, \eta (\mQ_m\otimes \mP_n)  \right)\\
     &=\sqrt{\frac{n\, m}{n+m}} \left(\mathbb{B}_{\C}  \left(e^{h_{n,m}}-e^{h_*}, \eta (\mQ\otimes \mP)  \right)+  \mathbb{B}_{\C}  \left(e^{h_{n,m}}-e^{h_*}, \eta (\mQ_m\otimes \mP_n-\mQ\otimes \mP)  \right)\right) \\
     &=\sqrt{\frac{n\, m}{n+m}} \mathbb{B}_{\C}  \left(e^{h_{n,m}}-e^{h_*}, \eta (\mQ\otimes \mP)  \right)+  o_\P(1).
 \end{align*}
Applying  Lemma~\ref{Lemma:Hadamard} and Lemma~\ref{Lemmainequqlity}  yields 
 $$A= \sqrt{\frac{n\, m}{n+m}} \mathbb{B}_{\C}  \left({h_{n,m}}-h_*, e^{h_*}\eta (\mQ\otimes \mP)  \right)+  o_\P(1). $$
Then the one dealing with the empirical process of $\mP$
\begin{align*}
    B
    & = \mathbb{B}_{\C}  \left(e^{h_*}, \eta (\mQ_m\otimes  v_{n,m}^\mP)  \right)\\
     \text{(by \eqref{BoundOnfixedhB}) } &= \mathbb{B}_{\C}  \left(e^{h_*}, \eta (\mQ\otimes  v_{n,m}^\mP )  \right)+o_\P(1)\\
     &=   ( v_{n,m}^\mP\otimes \mQ   )\left( \eta  \xi_{\mP, \mQ}   \right)+o_\P(1)\\
     &=    v_{n,m}^\mP( \eta_\x)+o_\P(1)
\end{align*}
and the last one with that of  $\mQ$, namely
$
    C_{m}= v_{n,m}^\mQ(\eta_{\y}).
$
In view of Theorem~\ref{Theorem:TCLpot},
we have 
\begin{multline}
    \label{limitA1}
    A=
\left\langle \eta,    (1-\AQ\AP)^{-1}\AQ {\bf i}_{\xi_{\mP, \mQ}}  v_{n,m}^\mP-(1-\AQ\AP)^{-1} {\bf i}_{ \xi_{\mQ, \mP}} v_{n,m}^\mQ\right\rangle_{L^2(\pi_{\mP,\mQ})}\\ + \left\langle \eta , \AP (1-\AQ\AP)^{-1} {\bf i}_{\xi_{\mQ, \mP}} v_{n,m}^\mQ -
 (1-\AP\AQ)^{-1} {\bf i}_{\xi_{\mP, \mQ}}  v_{n,m}^\mP\right\rangle_{L^2(\pi_{\mP,\mQ})}+o_\P(1).
\end{multline}
 In order  to introduce the term $\eta$ inside the operator, we observe that 
\begin{align*}
       \left\langle \eta,  (1-\AQ\AP)^{-1} {\bf i}_{ \xi_{\mQ, \mP}} v_{n,m}^\mQ\right\rangle_{L^2(\pi_{\mP,\mQ})}
       &=\langle \eta_{\x},(1-\AQ\AP)^{-1}{\bf i}_{ \xi_{\mQ, \mP}} v_{n,m}^\mQ\rangle_{L^2(\mP)}\\
 \text{(by \eqref{conjugate}) }      &=\langle (1-\AQ\AP)^{-1}\eta_{\x},{\bf i}_{ \xi_{\mQ, \mP}}v_{n,m}^\mQ\rangle_{L^2(\mP)}.
\end{align*}
 Each term of \eqref{limitA1} can be treated in the same way, which gives the relations;
%\pi_{\mP,\mQ}\bigg(\eta \left( \AP (1-\AQ\AP)^{-1} {\bf i}_{\mQ, \mP} (\mQ_m-\mQ) - (1-\AP\AQ)^{-1} {\bf i}_{\xi_{\mP, \mQ}}  (\mP_n-\mP)
\begin{align*}
      &\langle\eta,\AP (1-\AQ\AP)^{-1} {\bf i}_{ \xi_{\mQ, \mP}} v_{n,m}^\mQ\rangle_{L^2(\pi_{\mP,\mQ})}
       =\langle (1-\AQ\AP)^{-1}\AQ\eta_{\y},{\bf i}_{ \xi_{\mQ, \mP}} v_{n,m}^\mQ \rangle_{L^2(\mP)},\\
       &\langle \eta, (1-\AP\AQ)^{-1}{\bf i}_{ \xi_{\mQ, \mP}} v_{n,m}^\mP\rangle_{L^2(\pi_{\mP,\mQ})}
       =\langle (1-\AP\AQ)^{-1}\eta_{\y}, {\bf i}_{\xi_{\mP, \mQ}}  v_{n,m}^\mP \rangle_{L^2(\mQ)}\ \ {\rm and}\\
       &\langle \eta, (1-\AQ\AP)^{-1}\AQ  {\bf i}_{\xi_{\mP, \mQ}}  v_{n,m}^\mP\rangle_{L^2(\pi_{\mP,\mQ})}
       =\langle (1-\AP\AQ)^{-1}\AP\eta_{\x}, {\bf i}_{\xi_{\mP, \mQ}} v_{n,m}^\mP \rangle_{L^2(\mQ)}.
\end{align*}
Therefore, we obtain 
\begin{align*}
      A&= \langle (1-\AQ\AP)^{-1}(\AQ\eta_\y-\eta_{\x}),{\bf i}_{ \xi_{\mQ, \mP}}v_{n,m}^\mQ\rangle_{L^2(\mP)}\\
    &\qquad+ \langle (1-\AP\AQ)^{-1}(\AP\eta_{\x}-\eta_\y), {\bf i}_{\xi_{\mP, \mQ}} v_{n,m}^\mP \rangle_{L^2(\mQ)}\\
\text{(by \eqref{ruleApAq2}) }    &= v_{n,m}^\mQ \AP (1-\AQ\AP)^{-1}(\AQ\eta_\y-\eta_{\x})+ v_{n,m}^\mP \AQ (1-\AP\AQ)^{-1}(\AP\eta_{\x}-\eta_\y).
\end{align*}
 Adding the missing $B$ and $C$ terms and omitting  $o_\P(1)$ errors, we obtain that \eqref{thequqntity} is equal to 
\begin{align}\label{FirstandSecodn}
v_{n,m}^\mQ(\eta_\y+ \AP (1-\AQ\AP)^{-1}(\AQ\eta_\y-\eta_{\x}))+ v_{n,m}^\mP(\eta_\x+ \AQ (1-\AP\AQ)^{-1}(\AP\eta_{\x}-\eta_\y)). 
\end{align}
 Focusing  on the first term of \eqref{FirstandSecodn}, we compute 
\begin{align*} v_{n,m}^\mQ(\eta_\y+& \AP (1-\AQ\AP)^{-1}(\AQ\eta_\y-\eta_{\x}))\\
\text{(by \eqref{ruleApAq2}) } &=v_{n,m}^\mQ(\eta_\y+   (1-\AP\AQ)^{-1}\AP (\AQ\eta_\y-\eta_{\x}))
\\
& =v_{n,m}^\mQ(\eta_\y+   (1-\AP\AQ)^{-1} (\AP\AQ\eta_\y-\AP\eta_{\x}))\\
&=v_{n,m}^\mQ(\eta_\y+   (1-\AP\AQ)^{-1} (\eta_\y- \AP\eta_{\x}-(1-\AP\AQ)\eta_\y))\\
&=v_{n,m}^\mQ(  (1-\AP\AQ)^{-1} (\eta_\y- \AP\eta_{\x})).
\end{align*}
 Simplifying the second term of  \eqref{FirstandSecodn} in an identical fashion yields the result.

\section{Weak limit of the Divergences}\label{sec:divergences}

Recall that, for probabilities ${\rm P},{\rm Q}\in \mathcal{P}(\Omega)$, the quadratic Sinkhorn divergence \cite{GenPeyCut18} is defined as
$$D_{\epsilon}({\rm P},{\rm Q})= S_{\epsilon}({\rm P},{\rm Q})- \frac{1}{2}\left(  S_{\epsilon}({\rm P},{\rm P})+ S_{\epsilon}({\rm Q},{\rm Q}) \right).$$
 The correction terms in the definition of the Sinkhorn divergence are designed to debias the entropic optimal transport distance so that it becomes a \emph{bona fide} discrepancy measure (i.e., $D_{\epsilon}({\rm P},{\rm Q})=0$ if and only if ${\rm P}={\rm Q}$). 
This distance is of great interest for statistical applications.
 In this section, we study its limit distribution. 
Theorem~\ref{Theorem:TCLDiv} provides  limits of the quantity 
$a_n\left(D_{\epsilon}({\rm P}_n,{\rm Q})-D_{\epsilon}({\rm P}_n,{\rm Q})\right),$
where the sequence $\{a_n \}_{n\in \N}$ depends on the hypothesis  $H_0:\ {\rm P}={\rm Q}$ or $H_1:\ {\rm P} \neq  {\rm Q}$. In  the latter  case, the limit can be established by means of the \cite{delbarrio2019}'s technique based on Efron-Stein inequality---see also \cite{Weed19,delbarrio2021central} for reproduction and improvement of this argument---with common rate $a_n=\sqrt{n}$. Since the proof under $H_1$ does not have any technical inovation, we omit the details. 
The case ${\rm P}={\rm Q}$, however, has the faster rate $a_n=n$ and the limit depends on the $\mP$-Brownian bridge $\mathbb{G}_{\mP}$ in $(\mathcal{C}^{s}(\Omega))'$. 
More precisely, the limit in the one sample case will be the action of the random operator $\mathbb{G}_{\mP}$  on  the random function  $(1-\AP^2)^{-1}  {\bf i}_{\xi_{\mP,\mP}}(\mathbb{G}_{\mP})$.  
 To state our result, we define the bilinear form 
\begin{align*}
\mathbb{M}:\mathcal{S}\times \mathcal{S} &\to \R\\ 
      (\nu, \mu) &\mapsto  \nu((1-\AP^2)^{-1}  {\bf i}_{\xi_{\mP,\mP}}(\mu)),
\end{align*}
where $\mathcal{S}$ is defined in \eqref{OPeratorS}. 
$\mathbb{M}$ can be seen as the quadratic approximation of the Sinkhorn divergence under $H_0$. 
\begin{Theorem}\label{Theorem:TCLDivG}
Let $\Omega\subset{\R^d}$ be a compact set, ${\rm P}, {\rm Q}\in \mathcal{P}(\Omega)$ and ${\rm P}_n$ (resp. $\mQ_m$) be the empirical measure of the i.i.d. sample $\X_1,\dots,\X_n$ (resp. $\Y_1,\dots,\Y_m$) distributed as $\rm P$ (resp. $\mQ$). Then, if $m=m(n)\to \infty$ and $\frac{m}{n+m}\to \lambda\in (0,1)$, we have the following limits.
\begin{itemize}
\item Under $H_0:\ {\rm P} = {\rm Q},$
$$  D_\epsilon(\mP_n,\mQ_m)=\frac{\epsilon}{2}\, \mathbb{M}(\mP_n-\mQ_m, \mP_n-\mQ_m) + o_\P\left( \frac{n+m}{n\, m}\right)$$
and 
$  D_\epsilon(\mP_n,\mP)=\frac{\epsilon}{2}\, \mathbb{M}(\mP_n-\mP, \mP_n-\mP) + o_\P\left( \frac{1}{n}\right)$.
\item Under $H_1:\ {\rm P} \neq  {\rm Q},$
\begin{multline*}
     \sqrt{\frac{n\,m}{n+m}}(D_{\epsilon}(\mP_n,\mQ_m)- D_{\epsilon}(\mP,\mQ))= \int \psi_{\mP,\mQ}  d(\mP_n-\mP) + \int \psi_{\mQ,\mP}  d(\mQ_m-\mQ)\\+o_\P\left(\sqrt{\frac{n+m}{n\, m}} \right),
\end{multline*}
and $
    \sqrt{n}(D_{\epsilon}(\mP_n,\mQ)- D_{\epsilon}(\mP,\mQ))=\int \psi_{\mP,\mQ}  d(\mP_n-\mP) +o_\P\left(n^{-\frac{1}{2}} \right),
$
where $\psi_{\mP,\mQ} =f_{\mP,\mQ} - \frac{1}{2}(f_{\mP,\mP} +g_{\mP,\mP} ) $ and $\psi_{\mQ,\mP} =f_{\mQ,\mP} - \frac{1}{2}(f_{\mQ,\mQ} +g_{\mQ,\mQ} ) $.
\end{itemize}
\end{Theorem}

Using Theorem~\ref{Theorem:TCLDivG} to prove that $n\, D_\epsilon(\mP_n,\mP'_m)\xrightarrow{w}\mathbb{M}(\mathbb{G}_{\mP}, \mathbb{G}_{\mP})$ is trivial due to the continuity of all of the operators involving $\mathbb{M}$.   However, characterizing  in terms of known random variables is not immediate. We would like to use the following result. 
\begin{Lemma}[{\citealp[Theorem 6.32]{Grenander1968}}]\label{GaussianDecomp}
    A Gaussian random variable $\mathbb{G}$ in a separable Hilbert space $\mathcal{H}$ can be represented as the a.s. limit as $k\to \infty$ in $\mathcal{H}$ of the sum $ \sum_{i=1}^{k} \lambda_i N_i e_i$, where $\{N_i\}_{i=1}^k$ is an i.d.d. sequence of $N(0,1)$ and $\{e_i\}_{i=1}^k$ an orthonormal basis of $\mathcal{H}$. Therefore, $\|\mathbb{G}\|_{\mathcal{H}}^2=\sum_{i=1}^\infty \lambda_i^2 \chi(1)_i  $ where $\{ \chi(1)_i\}_{i\in \N}$ is an i.d.d sequence of Chi-squares with one degree of freedom. 
\end{Lemma}
   However, $(\mathcal{C}^{s}(\Omega))'$ is not Hilbertian or separable, and this is not immediately deduced. To write everything in terms of a inner product in a Hilbert space we use the following result. 
\begin{Lemma}
The operator $\mathbb{M}$
    is bilinear, definite positive and bounded. 
    As a consequence, the space $ (\mathcal{H}_{\mathbb{M}}, \mathbb{M} )  $ is a separable Hilbert space, where $\mathcal{H}_{\mathbb{M}}=\overline{\mathcal{S}}^{\mathbb{M}}$.
\end{Lemma} 
\begin{proof}
The fact that $\mathbb{M}$ is bilinear is straightforward. It is also bounded due to the following:
\begin{equation}\label{eq:boundedM}
     \nu((1-\AP^2)^{-1}  {\bf i}_{\xi_{\mP,\mP}}(\mu))\leq \|\nu\|_{(C^s(\Omega))'} \|\mu\|_{(C^s(\Omega))'} \| (1-\AP^2)^{-1}\|_{\Cs\to \Cs } \| {\bf i}_{\xi_{\mP,\mP}}\|_{(\Cs)'\to \Cs } ,
\end{equation}
where $\|\cdot \|_{\mathcal{B}\to \mathcal{G} }$ denotes the operator norm between the Banach spaces $\mathcal{B}$ and $ \mathcal{G}$. Note that $\| (1-\AP^2)^{-1}\|_{\Cs\to \Cs } \| {\bf i}_{\xi_{\mP,\mP}}\|_{(\Cs)'\to \Cs } $ is bounded due to Lemmas~\ref{Lemma:equivalent} and \ref{Lemma:ixi}. 
  We need to prove now that $\mathbb{M}$ is symmetric and positive for measures $\mu $ and $ \nu$. The complete statement holds by a density argument. First note that
  \begin{equation}
      \label{liberarIdent}
      (1-\AP^2 )^{-1} =1+\AP^2 (1-\AP^2)^{-1}.
  \end{equation}
The spectrum of the self-adjoint compact operator $\AP$ in  the Hilbert space 
 $ L^2_0(\mP)$
 is contained in $[0, 1-\delta]$ for some $\delta>0$  (see Lemma~\ref{Lemma:equivalent}). Since, self-adjoint compact operators in Hilbert spaces attain its norm  \cite[Proposition 6.9]{Brezis}, we have $\left\|\AP\right\|_{L^2_0(\mP)\to L^2_0(\mP)}<1$. Hence, 
  \begin{equation}
      \label{Noiman}
     \lim_{k\to \infty} \left\|(1-\AP^2 )^{-1}-\sum_{i=0}^{k}\AP^{2k} \right\|_{L^2_0(\mP)\to L^2_0(\mP)}=0.
 \end{equation} 
Set $f\in \mathcal{C}(\Omega)$.  By Fubinni's theorem, it holds  
   \begin{align*}
      \nu \AP^2 ( f)
      &= \int \xi_{\mP,\mP} ( {\bf u}, {\bf z}) \xi_{\mP,\mP} ({\bf z}, \y) f(\y) d\mP({\y}) \mP({\bf z}) d\nu({\bf u}) \\
      &= \langle   \AP f ,  {\bf i}_{\xi_{\mP,\mP}}(\nu) \rangle_{L^2(\mP)}. 
  \end{align*}
  By using \eqref{liberarIdent},  \eqref{Noiman} and \eqref{conjugate} we get
  \begin{align*}
      \mathbb{M}(\mu, \nu)&= \int \xi_{\mP, \mP} d(\mu \otimes \nu)+ \sum_{i=0}^{\infty} \langle   \AP^{2k+1} {\bf i}_{\xi_{\mP,\mP}}(\mu) ,  {\bf i}_{\xi_{\mP,\mP}}(\nu) \rangle_{L^2(\mP)} \\
      &=\int \xi_{\mP, \mP} d(\nu \otimes \mu)+ \sum_{i=0}^{\infty} \langle   \AP^{2k+1} {\bf i}_{\xi_{\mP,\mP}}(\nu) ,  {\bf i}_{\xi_{\mP,\mP}}(\mu) \rangle_{L^2(\mP)},
  \end{align*}
  from where we deduce the symmetry. 
Since, 
$ \xi_{\mP,\mP} ({\bf z}, \y) = \xi_{\mP,\mP} ({\bf y}, {\bf z})=e^{f_*(\y)} \mathbb{C}({\bf z}, \y) e^{f_*({\bf z})}$ and $\mathbb{C}$ is the  Gaussian kernel with bandwidth $\epsilon$,  the bilinear form $\mathbb{M}$ is positive.  

To prove the last claim we only need to show that $\mathcal{H}_{\mathbb{M}}$ is separable.  First we notice that \eqref{eq:boundedM} implies 
$\mathcal{H}_{\mathbb{M}}= \overline{\{\nu\in (\mathcal{C}(\Omega))': \quad \nu(\Omega)=0\}}^{\mathbb{M}}. $ Recall that, by the uniform boundedness principle, the sequence $\{d_n\}_{n\in \N}$ is bounded in $(\mathcal{C}(\Omega))'$. 
Since the weak$^*$ topology of $\mathcal{C}(\Omega))'$ is separable, 
we can find a countable dense $D\subset  \{\nu\in (\mathcal{C}(\Omega))': \quad \nu(\Omega)=0\}$ such that for any  $\mu\in  (\mathcal{C}(\Omega))'$ with $\mu(\Omega)=0$ there exists $d_n \xrightarrow{*} \mu$, with $d_n\in D$ for all $n$. Here the convergence $\xrightarrow{*} $ is the weak$^*$ in $(\mathcal{C}(\Omega))'$. Since the operator ${\bf i}_{\xi_{\mP, \mP}}$ is compact  (apply the Arzelà--Ascoli theorem to Lemma~\ref{Lemma:ixi}) and $(1-\AP^2)^{-1}$ is bounded, then 
$ (1-\AP^2)^{-1} {\bf i}_{\xi_{\mP, \mP}}(d_n-\mu)\xrightarrow{\Cso} 0 $. Therefore, $\mathbb{M}(d_n-\mu, d_n-\mu )\to 0$, which shows that $D$ is also dense in $\mathcal{H}_{\mathbb{M}}$. 
\end{proof}
    We can now use Lemma~\ref{GaussianDecomp} to get the following representation. 
\begin{Theorem}\label{Theorem:TCLDiv}
Let $\Omega\subset{\R^d}$ be a compact set, ${\rm P}\in \mathcal{P}(\Omega)$,  ${\rm P}_n$ and  $\mP_m'$ be independent empirical measures of  $\rm P$. Then, if $m=m(n)\to \infty$ and $\frac{m}{n+m}\to \lambda\in (0,1)$, we have the following limits:
   \begin{equation*}
         n\,D_1({\rm P}_n,{\rm P})\overset{w}{\longrightarrow} \frac{\epsilon}{2}\sum_{i=1}^{\infty}\lambda_i^2 N_j^2
   \end{equation*}
   and 
 $$
         {\frac{n\,m}{n+m}}D_{\epsilon}(\mP_n,\mP_m' )\xrightarrow{w}  \frac{\epsilon}{2}\sum_{i=1}^{\infty}\lambda_i^2 N_i^2,
 $$
where $ \{N_i\}_{i\in \N}$ is a sequence of i.i.d. random variables with $N_i\sim  N(0,1)$ and   $\{\lambda_{i}\}_{i\in \N}\subset [0, \infty)$ is such that $\sum_{i=1}^\infty\lambda_{i}^2< \infty. $ More precisely,  $\{\lambda_{i}\}_{i\in \N}\subset [0, \infty)$  are the eigenvalues of the covariance operator of $\mathbb{G}_P$ in the separable Hilbert space $\mathcal{H}_{\mathbb{M}}$. 
\end{Theorem}

\begin{proof}[{Proof of Theorem~\ref{Theorem:TCLDivG} if ${\rm P} = {\rm Q}$}]
As usually, we prove it in  the two-sample case. We denote $\mP'_m=\mQ_m=\frac{1}{m}\sum_{k=1}^m\delta_{\X'_k}$. We want to derive the limit of
\begin{align*}
     D_\epsilon({\rm P}_n,{\rm P}'_m)&= S_\epsilon({\rm P}_n,{\rm P}'_m)-{\textstyle \frac{1}{2}}\left(  S_\epsilon({\rm P}_n,{\rm P}_n)+ S_\epsilon({\rm P}'_m,{\rm P}'_m) \right)\\
   &={\textstyle \frac{1}{2}}\left(S_\epsilon({\rm P}_n,{\rm P}'_m)-S_\epsilon({\rm P}_n,{\rm P}_n)\right)+{\textstyle \frac{1}{2}}\left(  S_\epsilon({\rm P}_n,{\rm P}'_m)- S_\epsilon({\rm P}'_m,{\rm P}'_m) \right) .
\end{align*}
The following result writes the previous display in a way that we can apply Theorem~\ref{Theorem:TCLpot}. 

\begin{Lemma}\label{Lemma:TechnicalDiv}
    It holds that 
    \begin{multline}\label{developmentDivSeparateterms}
  D_\epsilon(\mP_n,\mP'_m)=\frac{1}{2}\int (g_{\mP_n,\mP_n}- g_{\mP'_m,\mP'_m})(d\mP'_m-d\mP'_n) \\
    +\frac{1}{4 \epsilon}\int ((h_{\mP_n,\mP'_m}-h_{\mP'_m,\mP'_m})^2 + (h_{\mP_n,\mP'_m}-h_{\mP_n,\mP_n})^2) d\pi_{\mP,\mP} + o_\P\left(\frac{n+m}{n\,m}\right).
\end{multline}
\end{Lemma}
There are two important terms in \eqref{developmentDivSeparateterms}; we linearize both of them of separately. \\

{\it First term of \eqref{developmentDivSeparateterms}:}
In view of Theorem~\ref{Theorem:TCLpot}, $\frac{(g_{\mP_n,\mP_n}-g_{\mP'_m,\mP'_m})}{\epsilon}$ can be expressed, up to additive $o_{\P}\left(\sqrt{\frac{n+ m}{n\,m}}\right)$ terms, in $\mathcal{C}^s(\Omega)$ as %\textcolor{blue}{It has the correct sign for $g$'s}
\begin{align*}
      \big((1-\AP^2)^{-1} {\bf i}_{\xi_{\mP,  \mP }} ( \mP_n-\mP'_m)&- (1-\AP^2)^{-1}\AP  {\bf i}_{\xi_{\mP,  \mP }} ( \mP_n-\mP'_m)\big)\\
    &=((1-\AP^2)^{-1}-(1-\AP^2)^{-1}\AP){\bf i}_{\xi_{\mP, \mP }} ( \mP_n-\mP'_m)\\
    &=((1-\AP^2)^{-1}(1-\AP)){\bf i}_{\xi_{\mP,  \mP }} ( \mP_n-\mP'_m).
\end{align*}
These $o_{\P}\left(\sqrt{\frac{n+ m}{n\,m}}\right)$ error terms in $\mathcal{C}^s(\Omega)$ become $o_{\P}\left({\frac{n+ m}{n\,m}}\right)$ when $(\mP_n-\mP'_m)$ acts over them. Hence, 
\begin{multline}
    \label{difpot}
    (\mP_n-\mP'_m) \left( g_{\mP'_m,\mP'_m}-g_{\mP_n,\mP_n}\right)\\
    = \epsilon \, (\mP_n-\mP'_m)  (1-\AP^2)^{-1}(1-\AP)({\bf i}_{\xi_{\mP, \mP}} ( \mP_n-\mP'_m))+o_{\P}\left(\frac{n+ m}{n\,m}\right).
\end{multline}

{\it Second term of \eqref{developmentDivSeparateterms}: }
Since, up to $o_{\P}\left(\sqrt{\frac{n+m}{n\,m}}\right)$ terms, we have 
\begin{multline*}
\frac{h_{\mP_n,\mP'_m}(\x,\y)-h_{\mP'_m,\mP'_m}(\x,\y)}{\epsilon}= -\left((1-\AP^2)^{-1}\AP {\bf i}_{\xi_{\mP, \mP}} ( \mP_n-\mP'_m)\right)(\x)\\+ \left((1-\AP^2)^{-1} {\bf i}_{\xi_{\mP, \mP}} ( \mP_n-\mP'_m)\right)(\y),
\end{multline*}
so that
\begin{multline}
    \label{aftersquqres}
   (h_{\mP_n,\mP'_m}(\x,\y)-h_{\mP'_m,\mP'_m}(\x,\y))^2\\
  = \epsilon^2\left( \left((1-\AP^2)^{-1}\AP {\bf i}_{\xi_{\mP, \mP}} ( \mP_n-\mP'_m)\right)(\x) \right)^2+ \epsilon^2\left( \left((1-\AP^2)^{-1} {\bf i}_{\xi_{\mP, \mP}} ( \mP_n-\mP'_m))\right)(\y) \right)^2
  \\ -2\epsilon^2\left((1-\AP^2)^{-1}\AP {\bf i}_{\xi_{\mP, \mP}}  (\mP_n- \mP'_m)\right)(\x) \left((1-\AP^2)^{-1} {\bf i}_{\xi_{\mP, \mP}}  (\mP_n- \mP'_m)\right)(\y).
\end{multline}
Integrating  \eqref{aftersquqres}  with respect to  $\pi_{\mP,\mP}$ yields
\begin{multline*}
  \epsilon^2 \| (1-\AP^2)^{-1}\AP {\bf i}_{\xi_{\mP, \mP}} ( \mP_n-\mP'_m)\|^2_{L^2(\mP)}+ \epsilon^2\|  (1-\AP^2)^{-1} {\bf i}_{\xi_{\mP, \mP}} ( \mP_n-\mP'_m)\|^2_{L^2(\mP)}\\
   -2\epsilon^2\langle (1-\AP^2)^{-1}\AP {\bf i}_{\xi_{\mP, \mP}} ( \mP_n-\mP'_m), \AP  (1-\AP^2)^{-1} {\bf i}_{\xi_{\mP, \mP}} ( \mP_n-\mP'_m) \rangle_{L^2(\mP)},
\end{multline*}
which is equal to 
\begin{equation} \label{aftersquqresInt}
    \epsilon^2\|  (1-\AP^2)^{-1} {\bf i}_{\xi_{\mP, \mP}} ( \mP_n-\mP'_m)\|^2_{L^2(\mP)}-\epsilon^2 \| (1-\AP^2)^{-1}\AP  {\bf i}_{\xi_{\mP, \mP}} ( \mP_n-\mP'_m)\|^2_{L^2(\mP)},
\end{equation}
    where we use that all the operators commute. 
 Expanding the squares and using commutativity again, we obtain 
\begin{multline*}
 \| (1-\AP^2)^{-1}\AP  {\bf i}_{\xi_{\mP, \mP}} ( \mP_n-\mP'_m)\|^2_{L^2(\mP)}\\
        =  \langle (1-\AP^2)^{-2}\AP^2  {\bf i}_{\xi_{\mP, \mP}} ( \mP_n-\mP'_m),{\bf i}_{\xi_{\mP, \mP}} ( \mP_n-\mP'_m)\ER \rangle_{L^2(\mP)} 
\end{multline*}
and 
\begin{multline*}
 \| (1-\AP^2)^{-1}  {\bf i}_{\xi_{\mP, \mP}} ( \mP_n-\mP'_m)\|^2_{L^2(\mP)}\\=  \langle (1-\AP^2)^{-2}{\bf i}_{\xi_{\mP, \mP}} ( \mP_n-\mP'_m),{\bf i}_{\xi_{\mP, \mP}} ( \mP_n-\mP'_m) \rangle_{L^2(\mP)}.
\end{multline*}
As a consequence, \eqref{aftersquqresInt} can be rewritten as
\begin{align}
\begin{split}
    \label{aftersquqresInt2}
    \epsilon^2 \langle &(1-\AP^2)^{-2}(1-\AP^2){\bf i}_{\xi_{\mP, \mP}} ( \mP_n-\mP'_m),{\bf i}_{\xi_{\mP, \mP}} ( \mP_n-\mP'_m) \rangle_{L^2(\mP)}\\
    &= \epsilon^2 \langle (1-\AP^2)^{-1}{\bf i}_{\xi_{\mP, \mP}} ( \mP_n-\mP'_m),{\bf i}_{\xi_{\mP, \mP}} ( \mP_n-\mP'_m) \rangle_{L^2(\mP)},\\
   % &= \epsilon^2 \langle (1-\AP^2)^{-1} {\bf i}_{\xi_{\mP, \mQ}} ( \mP_n-\mP'_m), {\bf i}_{\xi_{\mP, \mQ}} ( \mP_n-\mP'_m)\rangle_{L^2(\mP)}\\
  \text{(by \eqref{ruleApAq2})}  &= \epsilon^2  (\mP_n-\mP'_m) [\AP(1-\AP^2)^{-1} {\bf i}_{\xi_{\mP, \mP}} ( \mP_n-\mP'_m)],
\end{split}  
\end{align}
which turns out to be more convenient for the sequel. \\

{\it Derivation of the limit.} From Lemma~\ref{Lemma:TechnicalDiv} and equations  \eqref{difpot}  and \eqref{aftersquqresInt2}, we get (up to $o_{\P}\left(\frac{n+ m}{n\,m}\right)$ additive terms)
\begin{align*}
    D_\epsilon(\mP_n,\mP'_m)&= \frac{\epsilon}{2}  (\mP_n-\mP'_m)[  ((1-\AP^2)^{-1}{\bf i}_{\xi_{\mP, \mP}} ( \mP_n-\mP'_m)],
\end{align*}
which proves Theorem~\ref{Theorem:TCLDivG} under $H_0$. 
\end{proof}

\begin{funding}
    The research of Alberto González-Sanz and Jean-Michel Loubes is partially supported by the AI Interdisciplinary Institute ANITI, which is funded by
the French “Investing for the Future – PIA3” program under the Grant agreement ANR-19-PI3A-0004.  The research of  Jonathan Niles-Weed is partially funded by National Science Foundation, Grant DMS-2015291.
\end{funding}
\section*{Acknowledgements}
 Alberto González Sanz would like to thank Jonathan Niles-Weed for his hospitality at the Center for Data Science, New York University, where this research was concluded. The authors would like to thank Eustasio del Barrio for  commenting that his method based on the Efron-Stein inequality could work for the Sinkhorn Divergence under the alternative hypothesis and  Alex Kokot for pointing out some errors in the calculations on the earlier version of this work.  
 Alberto González Sanz would like to thank Gilles Mordant for helpful comments on the last version of this paper.

\appendix
%This supplementary material contains all of the omitted proofs of the main manuscript.  We start by proving some useful and probably well-known results of H\"older spaces.  Subsequently, we will employ these results to extend the bounds from \cite{del2022improved}. 
\section{Auxiliary results of H\"older spaces}
\begin{Lemma}\label{Lemma:Hadamard}
Let $\alpha\in \N$, $I\subset\R$ be a compact interval, $F\in \mathcal{C}^{\infty}(I)$ and $g\in \mathcal{C}^{\alpha}(\Omega)$, with $g(\Omega)\subset I$. Then the operator
\begin{align*}
    \delta_F:\mathcal{C}^{\alpha}(\Omega)&\longrightarrow  \mathcal{C}^{\alpha}(\Omega)\\
    g&\longmapsto F(g),
\end{align*}
is Fr\'echet differentiable in $g$ with derivative
$ D \delta_F(g)h=F'(g)h.$ Moreover, setting $f\in \mathcal{C}^{\alpha}(\Omega)$, the operator 
\begin{align*}
    \mathcal{C}^{\alpha}(\Omega)&\longrightarrow  \mathcal{C}^{\alpha}(\Omega)\\
    g&\longmapsto f g,
\end{align*}
is Fr\'echet differentiable in $g$ with derivative
$h\to f h.$

\end{Lemma}
\begin{proof}
To  verify  the first claim, i.e. 
$$\|F(g+h)-F(g)- F'(g)h\|_{\alpha}=o(\|h \|_{\alpha}),$$
we prove the bound for each derivative in a direction $ v  $ by a recursive argument on $ a = |v|$. Note that the first case, the uniform norm, is trivial; set $\x\in \Omega$ and apply Taylor’s theorem to obtain
$$  \rvert F(g(\x)+h(\x))-F(g(\x))- F'(g(\x))h(\x) \rvert \leq  \rvert F''(g(\x))h(\x)^2  \rvert \leq \sup_{I} \rvert F'' \rvert  \cdot  \| h\|_s. $$
Suppose that the result holds for $a\in \N$.
 Given a  derivative in a direction $b$ with $ \rvert b \rvert =a+1$,  there exists a decomposition $ D_{b}=D_{c}D_d$ with $ \rvert d \rvert =1$ and $ \rvert c \rvert =a$. Therefore, the chain rule yields
\begin{align*}
    &D_{b}(F(g+h)-F(g)- F'(g)h)\\
    &=D_{c}(F'(g+h)(D_d g+D_d h)-F'(g))D_d g- F''(g)D_d g h-F'(g) D_g h\\
    &=D_c\left((F'(g+h)-F'(g) -F''(g)h)D_d g+ (F'(g+h)-F'(g))D_d h\right).
\end{align*}
Since the function $F'$ still satisfies the assumptions of the theorem we obtain, by induction hypothesis and Lemma~\ref{BanachAlgebra}, the limits 
\begin{multline*}
     \| D_{c}((F'(g+h)-F'(g) -F''(g)h)D_d g)\|_{\infty}\\
     \leq C  \| D_{c}(F'(g+h)(D_d g+D_d h)-F'(g))\|_{\infty}\| g\|_{s}=o(\|h\|_s)
\end{multline*}
and   
$$
     \| D_{c}((F'(g+h)-F'(g))D_d h)\|_{\infty}\leq C  \| D_{c}(F'(g+h)-F'(g))\|_{\infty}\| h\|_{s}=o(\|h\|_s),
$$
which finish the proof of the first claim.
For the second one the relation 
$ f(g+h)-f g-f h=0 $ and Lemma~\ref{BanachAlgebra} conclude.
\end{proof}
\begin{Lemma}
\label{BanachAlgebra}
Let $f,g\in \mathcal{C}^s(\Omega)$ then there exists a constant $C$ depending on $\Omega$, $s$ and $d$ such that 
$$ \|f\, g\|_{s}\leq C\|f\|_s\|g\|_s. $$
\end{Lemma}
\begin{proof}
The  proof is direct consequence of the multivariate Leibniz rule, i.e.
\begin{equation*}
   \| D_{a}(f\,g)\|_{\infty}\leq  \sum_{ \rvert \alpha \rvert \leq a }{a\choose \alpha}\|D_{\alpha}f\|_{\infty}\|D_{a-\alpha}g\|_{\infty}\leq C\|f \|_{s}\|g \|_{s}.
\end{equation*}
\end{proof}
\section{Refined convergence rates of regularized potentials}\label{auxiliary}

The aim of this section is to refine the results of \cite{del2022improved} so that any errors that arise in the linearizations vanish rapidly enough. 
The proofs lack interest; they are merely calculations and Taylor expansions of the exponential function. 

\begin{Lemma}\label{Lemmainequqlity}
Let $\Omega\subset{\R^d}$ be a compact set, ${\rm P}, {\rm Q}\in \mathcal{P}(\Omega)$, and its  empirical versions $\mP_n,\mQ_m$, with $\frac{n}{m+n}\to \lambda\in (0,1)$. Then
\begin{align}\label{Inequqlitt1}
     &\| e^{h_{n,m}}- e^{h_{*}}-e^{h_{*}}(h_{n,m}-h_{*})\|_{\mathcal{C}^s(\Omega\times\Omega)} =o_{\P}\left(\sqrt{\frac{n+m}{n\, m}}\right), \quad \text{and}\\
     &\| e^{h_{n,m}}- e^{h_{*}}\|_{\mathcal{C}^s(\Omega\times\Omega)} =\mathcal{O}_{\P}\left(\sqrt{\frac{n+m}{n\, m}}\right).\label{Inequqlitt2}
\end{align}
\end{Lemma}
\begin{proof}
Since the functions $h_{n,m} $ and $h_{n,m}$ are uniformly bounded, for all $n,m\in \N$, Lemma~\ref{Lemma:Hadamard}  applied to the exponential gives
$$\|e^{h_{n,m}}- e^{h_{*}}-e^{h_{*}}(h_{n,m}-h_{*})\|_{s}=o_{\P}(\| h_{*}-h_{n,m}\|_s)$$
and, using Theorem 4.5 in \cite{del2022improved}, we obtain  \eqref{Inequqlitt1}. To prove \eqref{Inequqlitt2} we apply the inverse triangle inequality to \eqref{Inequqlitt1};
\begin{multline*}
    \| e^{h_{n,m}}- e^{h_{*}}-e^{h_{*}}(h_{n,m}-h_{*})\|_{\mathcal{C}^s(\Omega\times\Omega)}\\\geq \| e^{h_{n,m}}- e^{h_{*}}\|_{\mathcal{C}^s(\Omega\times\Omega)}- \| e^{h_{*}}(h_{n,m}-h_{*})\|_{\mathcal{C}^s(\Omega\times\Omega)}.
\end{multline*}
Then we apply Lemma \ref{BanachAlgebra} and \eqref{Inequqlitt1} to obtain
$$
   \| e^{h_{n,m}}- e^{h_{*}}\|_{\mathcal{C}^s(\Omega\times\Omega)} \leq C \| (h_{n,m}-h_{*})\|_{\mathcal{C}^s(\Omega\times\Omega)}+o_{\P}\left(\sqrt{\frac{n+m}{n\, m}}\right).
$$
Theorem 4.5 in \cite{del2022improved} concludes.
\end{proof}
\begin{Lemma}\label{Lemma_well_sep}
Let $\Omega\subset{\R^d}$ be a compact set and $\mP,\mQ\in \mathcal{P}(\Omega)$, then
$$ \left\rvert \frac{1}{\epsilon} S_\epsilon({\rm P},{\rm Q})-\int h d\mP d\mQ-\frac{1}{2}\int (h-h_{*})^2 d\pi_{\mP,\mQ} \right\rvert \leq \frac{1}{6} \|h-h_{*} \|_{\infty}^3 e^{\|h-h_{*} \|_{\infty}},$$
for all $h(\x,\y)=f(\x)+g(\y)$, with $f,g\in \mathcal{C}(\Omega)$ and $\int  e^{{h}}\C\, d({\rm P}\otimes {\rm Q})=1$.
\end{Lemma}
\begin{proof}
The application of Taylor's theorem to the exponential gives 
\begin{equation}
\label{Taylor_exp}
 \rvert e^{x}-1-x-\frac{1}{2}x^2 \rvert \leq \frac{1}{6} \rvert x^3 \rvert e^{ \rvert x \rvert }.
\end{equation}
Since also $\int  e^{h_{*}} \C\, d({\rm P}\otimes {\rm Q})=1$, \eqref{optimallityCodt0} yields
$$\frac{1}{\epsilon}S_\epsilon(\mP,\mQ)=\int h_{*}d\mP\mQ\\-\int\left(\, e^{{h_{*}}} -e^{h}\right)\C\, d({\rm P}\otimes {\rm Q}). $$

We can decompose
$$ -\int\left(\, e^{{h_{*}}} -e^{h}\right)\C\, d({\rm P}\otimes {\rm Q})=\int\left(\, e^{h-{h_{*}}} -1\right) e^{h_*}\C\, d({\rm P}\otimes {\rm Q})$$
and, finally, \eqref{Taylor_exp} implies
\begin{align*}
 \left\rvert \frac{1}{\epsilon}S_1(\mP,\mQ)-\int h_{*}d(\mP\otimes \mQ)-\int\left(h-h_{*}+\frac{1}{2}(h-h_{*})^2\right)d \pi_{\mP,\mQ}  \right\rvert \leq \frac{1}{6} \|h-h_{*} \|_{\infty}^3 e^{\|h-h_{*} \|_{\infty}}.
\end{align*}
Since  \eqref{optimallityCodt0} cancels the linear terms, the proof is completed.
\end{proof}

\begin{Lemma}\label{Lamma_exponential}
Let $\Omega\subset{\R^d}$ be a compact set, ${\rm P}, {\rm Q}\in \mathcal{P}(\Omega)$ and the associate empirical measures ${\rm P}_n$ and ${\rm Q}_m$. Then there exists a constant $C(\Omega,d,s)$ such that 
$$\left\|\left( e^{f_{n,m}(\x)}- e^{f_{*}(\x)}\right)\C(\x, \cdot)\right\|_s\leq C(\Omega,d,s) \|f_{n,m}-f_*\|_\infty,$$
for all $\x\in \Omega$
and 
$$ \E\left[ \left\|\left( e^{f_{n,m}(\X)}- e^{f_{*}(\X)}\right)\C(\X, \cdot)\right\|_s^2\right]\leq  C(\Omega,d,s)\left(\frac{1}{n}+\frac{1}{m}\right).$$
\end{Lemma}
\begin{proof}
The first claim is straightforward;
$$  \left\|\left( e^{f_{n,m}(\X)}- e^{f_{*}(\X)}\right)\C(\X, \cdot)\right\|_s \leq  \|e^{f_{n,m}}- e^{f_{*}}\|_{\infty} \left\|\C(\X, \cdot)\right\|_s\leq C\|e^{f_{n,m}}- e^{f_{*}}\|_{\infty} . $$
The second claim is thus completed by applying Theorem 4.5 in \cite{del2022improved}.
\end{proof}
\begin{Lemma}\label{Lemma:uniform_conv}
Let $\Omega\subset{\R^d}$ be a compact set, ${\rm P},\in \mathcal{P}(\Omega)$ and ${\rm P}_n$ be the empirical measure of the i.i.d. sample $\X_1,\dots,\X_n$ distributed as $\rm P$. Then 
$$ n\, \E\left[ \left(\sup_{\y\in \Omega, f\in \mathcal{C}^s(\Omega), \| f\|_s\leq 1}\left \rvert \int f(\x) \C(\x, \y) d(\mP_n-\mP)(\x)\right \rvert \right)^2 \right]=O(1). $$
Moreover, the class
$\{\x\mapsto g(\y)\C(\x, \y)f(\x), \quad \y\in \Omega, \ \|f\|_s\leq 1 , \ \|g\|_s\leq 1\} $
is $\mP$-Donsker.
\end{Lemma}
\begin{proof}
 The inclusion 
\begin{equation}
    \label{relation}
\left\{\x\mapsto g(\y)\C(\x, \y) f(\x), \quad \y\in \Omega, \ \|f\|_s\leq 1 , \ \|g\|_s\leq 1\right\}  \subset \left\{f\in \mathcal{C}^s(\Omega), \quad \|f\|_s\leq C\right\}
\end{equation}
 holds for certain constant $C>0$ and  \cite[2.7.2 Corollary]{Vart_Well} and \cite[Exercise 2.3.1]{Gin2015MathematicalFO} give
\begin{equation*}
    \E\left[\left(\sup_{\|f\|_s\leq 1}({\rm P}_n-{\rm P})(f)\right)^2\right]\leq \frac{C}{n}.
\end{equation*}
Therefore the first statement holds. The last one is consequence of \eqref{relation},  Theorems~2.5.2 and 2.7.1  in \cite{Vart_Well}
\end{proof}

%%%%%

\section{Proofs of the Lemmas}\label{sec:proofLemmas}
\begin{proof}[Proof of Lemma~\ref{Lemma:ixi}]
We prove $ \partial_b {\bf i}_{h}(\nu)={\bf i}_{\partial_{b, \y } h} (\nu) \in \mathcal{C}(\Omega) $ inductively in $\beta=|b|\in \N$ the result.  For $\beta=0$ we only need to check the continuity, which holds by the same steps of \eqref{eq:contder} below. We avoid repeated arguments and skip the proof. Assume that $ \partial_{b'} {\bf i}_{h}(\nu)={\bf i}_{\partial_{b', \y } h} (\nu) $ for all $|b'|\leq \beta$. 
Let $ b=(b_1, \dots, b_d) $ be a multi-index with $|b|\leq\beta+1 $.   Assume without loss of generality that $b_1\geq 1$ and consider $b'=(b_1-1, b_2,\dots, b_d) $ which obviously satisfies  $|b'|\leq \beta $. Therefore, 
\begin{align*}
 \frac{{\bf i}_{\partial_{b', \y } h} (\nu)(\y+t{\bf e}_1)-  {\bf i}_{\partial_{b', \y } h} (\nu)(\y)}{t}  &= \frac{1}{t} \nu\left( \partial_{b', \y } h(\cdot, \y+ t{\bf e}_1 ) -\partial_{b', \y } h(\cdot, \y)  \right)\\
 &= \nu\left(  \frac{\partial_{b', \y } h(\cdot, \y+ t{\bf e}_1 ) -\partial_{b', \y } h(\cdot, \y)}{t}   \right).
\end{align*}
Due to the limit 
$$ \lim_{t\to 0}\left\| \frac{\partial_{b', \y } h(\cdot, \y+ t{\bf e}_1 ) -\partial_{b', \y } h(\cdot, \y)}{t} - \partial_{b, \y } h(\cdot, \y)\right\|_{\mathcal{C}^{\alpha}(\Omega)}=0$$
we get 
$$ \lim_{t\to 0}\frac{{\bf i}_{\partial_{b', \y } h}^{\alpha} (\nu)(\y+t{\bf e}_1)-  {\bf i}_{\partial_{b', \y } h}^{\alpha} (\nu)(\y)}{t}= {\bf i}_{\partial_{b, \y } h}^{\alpha} (\nu)(\y), \quad \forall\, \y\in \Omega . $$
Since 
\begin{align}
\begin{split}
    \label{eq:contder}
     \vert  {\bf i}_{\partial_{b, \y } h}^{\alpha} (\nu)(\y')- {\bf i}_{\partial_{b, \y } h}^{\alpha} (\nu)(\y)\vert &=\nu\left(  {\partial_{b, \y } h(\cdot, \y' ) -\partial_{b, \y } h(\cdot, \y)}   \right) \\
     &\leq \|\nu\|_{(\mathcal{C}^{\alpha}(\Omega))'} \| \partial_{b, \y } h(\cdot, \y' ) -\partial_{b, \y } h(\cdot, \y)\|_{\mathcal{C}^{\alpha}(\Omega)}
\end{split}
\end{align}
and $h\in \mathcal{C}^{\infty}(\Omega^2)$, we get $ {\bf i}_{\partial_{b, \y } h}^{\alpha} (\nu)\in \mathcal{C}(\Omega
)$. As a consequence, $ \partial_b {\bf i}_{h}(\nu)={\bf i}_{\partial_{b, \y } h}^{\alpha} (\nu) \in \mathcal{C}(\Omega) $ for all $b\in \N^d$ and, of course, ${\bf i}_{h}(\nu)\in \mathcal{C}^{\infty}(\Omega)$. 

%The second claim holds due to 
%$ |\partial_b {\bf i}_{h}(\nu)(\y)| \leq \|h\|_{\mathcal{C
%}^{\alpha+|b|}}\|\nu\|_{(\mathcal{C}^{\alpha}(\Omega))'}.$
\end{proof}

\begin{proof}[Proof of Lemma~\ref{Lemma:TechnicalDiv}]
    First note that, since 
$\int h_{\mP'_m,\mP'_m}\C \, d(\mP_n \otimes \mP'_m)=1 $,
Lemma~\ref{Lemma_well_sep} yields
$$
  S_\epsilon({\rm P}_n,{\rm P}'_m)= \int h_{\mP'_m,\mP'_m} d(\mP_n \otimes \mP'_m)+\frac{1}{2 \epsilon}\int (h_{\mP_n,\mP'_m}-h_{\mP'_m,\mP'_m})^2 d\pi_{\mP_n,\mP'_m} +o_\P\left(\frac{n+m}{n\,m}\right)
   ,
$$
where the last term is consequence of Lemma~\ref{BanachAlgebra} and Theorem~4.5 in \cite{del2022improved}. 
Therefore 
\begin{multline*}
    S_\epsilon({\rm P}_n,{\rm P}'_m)= S_\epsilon({\rm P}'_{ m},{\rm P}'_m)+\int h_{\mP'_m,\mP'_m} d((\mP_{n}-\mP_m') \otimes \mP'_m)\\
    +\frac{1}{2\epsilon}\int (h_{\mP_n,\mP'_m}-h_{\mP'_m,\mP'_m})^2 d\pi_{\mP_n,\mP'_m} +o_\P\left(\frac{n+m}{n\,m}\right).
\end{multline*}
By the same argument we also have 
\begin{multline*}
    S_\epsilon({\rm P}_n,{\rm P}'_m)= S_\epsilon({\rm P}_n,{\rm P}_n)+\int h_{\mP_n,\mP_n}  d(\mP_n \otimes (\mP'_m-\mP_n) )\\
    +\frac{1}{2\epsilon}\int (h_{\mP_n,\mP'_m}-h_{\mP_n,\mP_n})^2 d\pi_{\mP_n,\mP'_m} +o_\P\left(\frac{n+m}{n\,m}\right).
\end{multline*}
and, as a consequence, 
\begin{multline*}
  D_\epsilon(\mP_n,\mP'_m)=\frac{1}{2}\left(\int h_{\mP'_m,\mP'_m} d((\mP_n-\mP_m') \otimes \mP'_m)+\int h_{\mP_n,\mP_n} d(\mP_n \otimes (\mP'_m-\mP_n) )\right) \\
    +\frac{1}{4 \epsilon}\int ((h_{\mP_n,\mP'_m}-h_{\mP'_m,\mP'_m})^2 + (h_{\mP_n,\mP'_m}-h_{\mP_n,\mP_n})^2) d\pi_{\mP_n,\mP'_m} + o_\P\left(\frac{n+m}{n\,m}\right).
\end{multline*}
The relations
$$ \int h_{\mP'_m,\mP'_m} d((\mP_n-\mP_m') \otimes \mP'_m)= \int f_{\mP_n,\mP_n} (d\mP_n-d\mP'_m), $$
$$ \int h_{\mP_n,\mP_n}  d(\mP_n \otimes (\mP'_m-\mP_n) )= \int g_{\mP'_m,\mP'_m} (d\mP_m'-d\mP_n) $$
and the fact that, in this symmetric case, $ f_{\mP_n,\mP_n}= g_{\mP_n,\mP_n}+c$, give
\begin{multline*}
  D_\epsilon(\mP_n,\mP'_m)=\frac{1}{2}\int (g_{\mP_n,\mP_n}- g_{\mP'_m,\mP'_m})(d\mP'_m-d\mP'_n) \\
    +\frac{1}{4 \epsilon}\int ((h_{\mP_n,\mP'_m}-h_{\mP'_m,\mP'_m})^2 + (h_{\mP_n,\mP'_m}-h_{\mP_n,\mP_n})^2) d\pi_{\mP_n,\mP'_m} + o_\P\left(\frac{n+m}{n\,m}\right).
\end{multline*}
Since $\|\pi_{\mP_n,\mP'_m}-\pi_{\mP,\mP}\|_{(\mathcal{C}^{ s }(\Omega^2))'}\rightarrow 0$ a.s. and $\|\big( h_{\mP_n,\mP'_m}-h_{\mP'_m,\mP'_m}\big)^2\|_{s}=\mathcal{O}_{\P}\left(\frac{n+m}{n\,m}\right), $
the inequality 
\begin{multline*}
    \left\rvert\int (h_{\mP_n,\mP'_m}-h_{\mP'_m,\mP'_m})^2 (d\pi_{\mP_n,\mP'_m}-d\pi_{\mP,\mP}) \right\rvert\\
	\leq\| (h_{\mP_n,\mP'_m}-h_{\mP'_m,\mP'_m})^2\|_{ s }\| \pi_{\mP_n,\mP'_m}-\pi_{\mP,\mP}\|_{(\mathcal{C}^{ s }(\Omega^2))'}
\end{multline*}
 yields the result. 
\end{proof}
\bibliographystyle{imsart-number}
\bibliography{references2}
\end{document}